\documentclass[11pt]{amsart}
\usepackage{amsmath}
\usepackage{latexsym,amsfonts,amssymb,mathrsfs}
\usepackage{mathdots}
\usepackage{color}
\usepackage [all,2cell,color]{xy}
\usepackage{enumitem} 
\usepackage{bbm}
\usepackage{tikz}
\usepackage{tikz-cd}
\usepackage[T1]{fontenc}
 \usepackage[utf8]{inputenc}
 \usepackage{blkarray} 
 \usepackage{hyperref}
\usepackage[capitalise]{cleveref}

\UseAllTwocells
\usepackage[
textwidth=3cm,
textsize=small,
colorinlistoftodos]
{todonotes}

\usetikzlibrary{positioning} 
\usetikzlibrary{decorations.pathreplacing}
\usetikzlibrary{arrows, decorations.markings} 
\pgfarrowsdeclarecombine{twotriang}{twotriang}{stealth}{stealth}%
{stealth}{stealth}
\tikzset{arrow/.style={-stealth}}
\tikzset{arrowshorter/.style={-stealth, shorten <=2pt, shorten >=2pt}}
\tikzset{arrowmuchshorter/.style={-stealth, shorten <=7pt, shorten >=6pt}}
\tikzset{mono/.style={>-stealth}} 
\tikzset{epi/.style={-twotriang}} 
\tikzset{twoarrowlonger/.style={double,double distance=1.5pt,
shorten <=5pt,shorten >=6pt,
decoration={markings,mark=at position -4pt with {\arrow[scale=1.75]{>}}},
preaction={decorate}}} 

\tikzset{twoarrow/.style={double,double distance=1.5pt,
shorten <=6pt,shorten >=7pt, 
decoration={markings,mark=at position -4pt
with {\arrow[scale=1.75]{>}}},
preaction={decorate} 
}
}
\tikzset{%
    symbol/.style={%
        draw=none,
        every to/.append style={%
            edge node={node [sloped, allow upside down, auto=false]{$#1$}}}
    }
}

\tikzset{mapstikz/.style={-stealth, 
decoration={markings,mark=at position 0pt with {\arrow[scale=0.5]{|}}}, preaction={decorate}}}

\theoremstyle{plain}   
\newtheorem{thm}{Theorem}[section] 
\makeatletter\let\c@thm\c@thm\makeatother
\newtheorem{cor}{Corollary}[section]
\makeatletter\let\c@cor\c@thm\makeatother
\newtheorem{lem}{Lemma}[section]
\makeatletter\let\c@lem\c@thm\makeatother
\newtheorem{prop}{Proposition}[section]
\makeatletter\let\c@prop\c@thm\makeatother

\makeatletter\let\c@claim\c@thm\makeatother

\newtheorem{conjecture}{Conjecture}[section]
\makeatletter\let\c@conjecture\c@thm\makeatother

\newtheorem*{unnumberedtheoremA}{Theorem A}
\newtheorem*{unnumberedtheoremB}{Theorem B}
\newtheorem*{unnumberedtheoremC}{Theorem C}

\theoremstyle{definition}

\newtheorem{defn}{Definition}[section]
\makeatletter\let\c@defn\c@thm\makeatother
\newtheorem{const}{Construction}[section]
\makeatletter\let\c@const\c@thm\makeatother
\newtheorem{notn}{Notation}[section]
\makeatletter\let\c@notn\c@thm\makeatother

\theoremstyle{remark}

\newtheorem{rmk}{Remark}[section]
\makeatletter\let\c@rmk\c@thm\makeatother

\makeatletter\let\c@ex\c@thm\makeatother

\makeatletter\let\c@observation\c@thm\makeatother

\makeatletter\let\c@warning\c@thm\makeatother
\newtheorem{digression}{Digression}[section]
\makeatletter\let\c@warning\c@thm\makeatother

\makeatletter
\let\c@equation\c@thm
\numberwithin{equation}{section}
\makeatother

\newcommand{\newrefformat}[2]{}

\crefname{lem}{Lemma}{Lemmas}
\crefname{thm}{Theorem}{Theorems}
\crefname{defn}{Definition}{Definitions}
\crefname{notn}{Notation}{Notations}
\crefname{const}{Construction}{Constructions}
\crefname{prop}{Proposition}{Propositions}
\crefname{rmk}{Remark}{Remarks}
\crefname{cor}{Corollary}{Corollaries}
\crefname{equation}{Display}{Displays}
\crefname{ex}{Example}{Examples}

\newcommand{\cC}{\mathcal{C}}

\newcommand{\cO}{\mathcal{O}}

\newcommand{\cS}{\mathcal{S}}

\newcommand{\cat}{\cC\!\mathit{at}}

\newcommand{\set}{\cS\!\mathit{et}}

\newcommand{\sset}{\mathit{s}\set}

\newcommand{\psh}[1]{\set^{#1^{\op}}}

\newcommand{\pstrat}{\psh{t\Delta}}

 \newcommand{\twocat}{2\cat}

\DeclareMathOperator{\Hom}{Hom}
\DeclareMathOperator{\Map}{Map}

    \newcommand{\Nstreet}{N^{Street}}
      \newcommand{\hostreet}{c^{Street}}

        \newcommand{\Ntdelta}{N^{RS}}
        \newcommand{\hotdelta}{c^{RS}}
                
        \newcommand{\Nnat}{N^{\natural}}
    \newcommand{\honat}{c^{\natural}}

\DeclareMathOperator{\id}{id}

\DeclareMathOperator{\Ob}{Ob}
\newcommand{\aamalg}[1]{\underset{#1}{{\amalg}}} 
\DeclareMathOperator{\op}{op}

\newcommand{\eqDelta}{\Delta[3]_{eq}}

\usepackage{pigpen}
\newcommand{\po}{\ar@{}[dr]|{\text{\pigpenfont R}}}
\newcommand{\pb}{\ar@{}[dr]|{\text{\pigpenfont J}}}

\newcommand\simpf[9]{%
    \def\tempa{#1}%
    \def\tempb{#2}%
    \def\tempc{#3}%
    \def\tempd{#4}%
    \def\tempe{#5}%
    \def\tempf{#6}%
    \def\tempg{#7}%
    \def\temph{#8}%
    \def\tempi{#9}%
}
\newcommand\simpfcontinued[1]{%
    \begin{tikzcd}[column sep=1.3cm, row sep=0.4cm, baseline=(current  bounding  box.center), ampersand replacement=\&]
 3 \& 2 \arrow[l, "\tempc" swap] \&[-7mm] \&[-7mm] 3 \& 2 \arrow[l, "\tempc" swap]\\
 {} \& \&[-7mm] {=} \&[-7mm] \& \\
 0 \arrow[r, "\tempa" swap]\arrow[uu, "\tempd", ""{name=a031,inner sep=2pt, swap} ] \arrow[uur, "\tempe"{near start, xshift=0.15cm}, ""{name=a02,inner sep=2pt, swap}]\& 1 \arrow[uu, "\tempb" swap]\&[-7mm] \&[-7mm] 0 \arrow[uu, "\tempd", ""{name=a032,inner sep=2pt, swap}]\arrow[r, "\tempa" swap]\& 1\arrow[uu, "\tempb" swap] \arrow[uul, "\tempf"{near end, xshift=0.25cm}, ""{name=a13,inner sep=2pt, swap}]
 \arrow[Rightarrow, from=a031, to=1-2, shorten >= 0.3cm, "\tempg" ]
 \arrow[Rightarrow, from=a02, to=3-2, shorten >= 0.3cm, "\temph" swap ]
 \arrow[Rightarrow, from=a032, to=3-5, shorten >= 0.3cm, "\tempi" swap ]
 \arrow[Rightarrow, from=a13, to=1-5, shorten >= 0.3cm, "#1" ]
 \end{tikzcd}
}

\newcommand\simpfbdrya[9]{%
    \def\tempa{#1}%
    \def\tempb{#2}%
    \def\tempc{#3}%
    \def\tempd{#4}%
    \def\tempe{#5}%
    \def\tempf{#6}%
    \def\tempg{#7}%
    \def\temph{#8}%
    \def\tempi{#9}%
}
\newcommand\simpfbdryacontinued[1]{%
    \begin{tikzcd}[column sep=1.3cm, row sep=0.4cm, baseline=(current  bounding  box.center), ampersand replacement=\&]
 3 \& 2 \arrow[l, "\tempc" swap] \&[-7mm] \&[-7mm] 3 \& 2 \arrow[l, "\tempc" swap]\\
 {} \& \&[-7mm] {\mbox{ and }} \&[-7mm] \& \\
 0 \arrow[r, "\tempa" swap]\arrow[uu, "\tempd", ""{name=a031,inner sep=2pt, swap} ] \arrow[uur, "\tempe"{near start, xshift=0.15cm}, ""{name=a02,inner sep=2pt, swap}]\& 1 \arrow[uu, "\tempb" swap]\&[-7mm] \&[-7mm] 0 \arrow[uu, "\tempd", ""{name=a032,inner sep=2pt, swap}]\arrow[r, "\tempa" swap]\& 1\arrow[uu, "\tempb" swap] \arrow[uul, "\tempf"{near end, xshift=0.25cm}, ""{name=a13,inner sep=2pt, swap}]
 \arrow[Rightarrow, from=a031, to=1-2, shorten >= 0.3cm, "\tempg" ]
 \arrow[Rightarrow, from=a02, to=3-2, shorten >= 0.3cm, "\temph" swap ]
 \arrow[Rightarrow, from=a032, to=3-5, shorten >= 0.3cm, "\tempi" swap ]
 \arrow[Rightarrow, from=a13, to=1-5, shorten >= 0.3cm, "#1" ]
 \end{tikzcd}
}

\newcommand\simpftri[9]{%
    \def\tempa{#1}%
    \def\tempb{#2}%
    \def\tempc{#3}%
    \def\tempd{#4}%
    \def\tempe{#5}%
    \def\tempf{#6}%
    \def\tempg{#7}%
    \def\temph{#8}%
    \def\tempi{#9}%
}
\newcommand\simpftricontinued[1]{%
    \begin{tikzcd}[column sep=1.3cm, row sep=0.4cm, baseline=(current  bounding  box.center), ampersand replacement=\&]
 3 \& 2 \arrow[l, "\tempc" swap] \&[-7mm] \&[-7mm] 3 \& 2 \arrow[l, "\tempc" swap]\\
 {} \& \&[-7mm] {\Rrightarrow} \&[-7mm] \& \\
 0 \arrow[r, "\tempa" swap]\arrow[uu, "\tempd", ""{name=a031,inner sep=2pt, swap} ] \arrow[uur, "\tempe"{near start, xshift=0.15cm}, ""{name=a02,inner sep=2pt, swap}]\& 1 \arrow[uu, "\tempb" swap]\&[-7mm] \&[-7mm] 0 \arrow[uu, "\tempd", ""{name=a032,inner sep=2pt, swap}]\arrow[r, "\tempa" swap]\& 1\arrow[uu, "\tempb" swap] \arrow[uul, "\tempf"{near end, xshift=0.25cm}, ""{name=a13,inner sep=2pt, swap}]
 \arrow[Rightarrow, from=a031, to=1-2, shorten >= 0.3cm, "\tempg" ]
 \arrow[Rightarrow, from=a02, to=3-2, shorten >= 0.3cm, "\temph" swap ]
 \arrow[Rightarrow, from=a032, to=3-5, shorten >= 0.3cm, "\tempi" swap ]
 \arrow[Rightarrow, from=a13, to=1-5, shorten >= 0.3cm, "#1" ]
 \end{tikzcd}
}

\newcommand\simpfver[9]{%
    \def\tempa{#1}%
    \def\tempb{#2}%
    \def\tempc{#3}%
    \def\tempd{#4}%
    \def\tempe{#5}%
    \def\tempf{#6}%
    \def\tempg{#7}%
    \def\temph{#8}%
    \def\tempi{#9}%
}
\newcommand\simpfvercontinued[5]{%
    \begin{tikzcd}[column sep=1.3cm, row sep=0.4cm, baseline=(current  bounding  box.center), ampersand replacement=\&]
 #5 \& #4 \arrow[l, "\tempc" swap] \&[-7mm] \&[-7mm] #5 \& #4 \arrow[l, "\tempc" swap]\\
 {} \& \&[-7mm] {=} \&[-7mm] \& \\
 #2 \arrow[r, "\tempa" swap]\arrow[uu, "\tempd", ""{name=a031,inner sep=2pt, swap} ] \arrow[uur, "\tempe"{near start, xshift=0.15cm}, ""{name=a02,inner sep=2pt, swap}]\& #3 \arrow[uu, "\tempb" swap]\&[-7mm] \&[-7mm] #2 \arrow[uu, "\tempd", ""{name=a032,inner sep=2pt, swap}]\arrow[r, "\tempa" swap]\& #3\arrow[uu, "\tempb" swap] \arrow[uul, "\tempf"{near end, xshift=0.25cm}, ""{name=a13,inner sep=2pt, swap}]
 \arrow[Rightarrow, from=a031, to=1-2, shorten >= 0.3cm, "\tempg" ]
 \arrow[Rightarrow, from=a02, to=3-2, shorten >= 0.3cm, "\temph" swap ]
 \arrow[Rightarrow, from=a032, to=3-5, shorten >= 0.3cm, "\tempi" swap ]
 \arrow[Rightarrow, from=a13, to=1-5, shorten >= 0.3cm, "#1" ]
 \end{tikzcd}
}

\author{Viktoriya Ozornova}
\address{Fakult\"at f\"ur Mathematik, Ruhr-Universit\"at Bochum, D-44780 Bochum, Germany}
\email{viktoriya.ozornova@rub.de}

\author{Martina Rovelli}
\address{Department of Mathematics,
Johns Hopkins University,
Baltimore, MD 21218, USA
}
\email{mrovelli@math.jhu.edu}

\thanks{
The second-named author was partially funded by the Swiss National Science Foundation, grant P2ELP2\textunderscore172086. 
}

\begin{document}

\title{Nerves of $2$-categories and $2$-categorification of $(\infty,2)$-categories}

\maketitle

\begin{abstract}
We show that the homotopy theory of strict $2$-categories embeds in that of $(\infty,2)$-categories in the form of $2$-precomplicial sets.
More precisely, we construct a nerve-categorification adjunction that is a Quillen pair between Lack's model structure for $2$-categories and Riehl--Verity's model structure for $2$-complicial sets. Furthermore, we show that Lack's model structure is transferred along this nerve and that the nerve is homotopically fully faithful.
\end{abstract}

\tableofcontents

\section*{Introduction}
Higher categories turned out to be the right language to formalize many different mathematical phenomena. 
A category of higher order consists not only of objects and morphisms between objects, but also of morphisms in any dimension between morphisms of lower dimension.

The first incarnation of higher category theory was the study of strict $n$-categories, in which there are morphisms up to dimension $n$ and the composition of morphisms of any order is strictly associative and unital. An established way to formalize this idea is to consider categories enriched over the category of $(n-1)$-categories. There is a natural notion of equivalence of $n$-categories, given by $n$-functors that are homotopically essentially surjective on objects and induce equivalences on every hom-$(n-1)$-category, and there is interest in understanding the homotopy theory of $n$-categories up to equivalence. To this end, model structures for the category of $n$-categories have been established.

To accomodate many examples that appear in nature, more recently higher category theory has included also the study of various forms of weak higher categories, often referred to as $(\infty, n)$-categories.  
In an $(\infty,n)$-category there are higher morphisms at any level, and they are invertible in dimension higher than $n$. The composite of two morphisms is only defined only up to a higher morphism, associativity and unitality only hold up to a higher morphism, and invertibility of morphisms is expressed up to a higher morphism.
Unlike for strict $n$-categories, there is a lot of freedom in interpreting which mathematical object should implement the concept of an $(\infty,n)$-category. As of today, there are several models for $(\infty,n)$-categories, together with model structures that encode their homotopy theory.

For $n=0,1$, amongst many the most established models for $(\infty,n)$-categories there are Kan complexes and quasi-categories. These objects have both the nature of simplicial sets satisfying certain lifting properties. In this case, the $0$-simplices are thought of as objects, $m$-simplices correspond to $m$-morphisms, and the simplicial structure encodes identities and composition law of an $(\infty,0)$- or $(\infty,1)$-category.
For $n>1$, in this paper we use a model of $(\infty,n)$-categories given by $n$-(pre)complicial sets. These are simplicial sets having a certain lifting property in which some of the simplices that behave as equivalences are marked. To formalize this idea, model structures for $n$-(pre)complicial sets have been established.

Given the simplicial nature of these models, there is a standard way to regard a strict $n$-category as an $(\infty,n)$-category, by looking at its Street nerve.
The purpose of this paper is to show how this nerve can be used to meaningfully embed the homotopy theory of $2$-categories in that of $(\infty,2)$-categories, in the form of $2$-precomplicial sets.
While we go through the comparison of $n$-categories and $(\infty,n)$-categories with $n=0,1$ to then focus on $n=2$, we also indicate how the techniques employed are likely to generalize to the case $n>2$.

We construct a Quillen pair between the Riehl--Verity's model structure for $(\infty,2)$-categories and Lack's model structure for $2$-categories, and show that the former is right-transferred from the latter. This formalizes the idea that $(\infty,2)$-categories generalize $2$-categories, and that the homotopy theory of $2$-cate\-gories completely embeds in that of $(\infty,2)$-categories. Similar ideas working with $(\infty,2)$-categories in the form of $\Theta_2$-sets are pursued in \cite{campbell}. 

Given a $2$-category $\cC$, Street \cite{StreetOrientedSimplexes} and later Duskin \cite{duskin} defined the nerve $\Nstreet_2(\cC)$ to be a $3$-coskeletal simplicial set in which a $0$-simplex is an object $x$ of $\cC$, a $1$-simplex is a $1$-morphism $f\colon x\to y$ of $\cC$, a $2$-simplex is a $2$-morphism $\alpha\colon f\Rightarrow hg$ of $\cC$, and a $3$-simplex consists of four $2$-morphisms of $\cC$ for which the following composites agree.
\[
\simpfver{d}{c}{e}{a}{b}{f}{}{}{}\simpfvercontinued{}{x}{y}{z}{w}
\]
Unlike its $1$-dimensional analog, the Street--Duskin nerve is not fully faithful as a functor $\Nstreet_2\colon2\cat\to\psh{\Delta}$. This issue is a manifestation of the fact that composition of $2$-morphisms is no longer completely captured by the simplicial structure.

To correct the issue, Roberts and later Street endowed the Street--Duskin nerve with further structure, by marking the simplices of the nerve corresponding to identity morphisms in $\cC$.
The nature of the Roberts--Street nerve is that of a \emph{stratified simplicial set}. Stratified simplicial sets are certain presheaves over an enlargement $t\Delta$ of the usual simplex category $\Delta$.
Verity \cite{VerityComplicialAMS} shows that the Roberts--Street nerve is indeed fully faithful as a functor $\Ntdelta_2\colon2\cat\to\psh{t\Delta}$.

The corresponding adjunction
$$\hotdelta_2\colon\psh{t\Delta}\rightleftarrows 2\cat\colon\Ntdelta_2,$$
however, does not have good homotopical properties when endowing $2\cat$ with the Lack model structure \cite{lack1} and $\psh{t\Delta}$ with the Riehl--Verity model structure for $(\infty,2)$-categories \cite{or}.
For instance, not all nerves of $2$-categories are $(\infty,2)$-categories, and the adjunction is not a Quillen pair.

Rather than marking only the simplices coming from identities, we consider a different marking on the Street--Duskin nerve $\Nstreet_2(\cC)$. We mark all $1$-simplices corresponding to equivalences in $\cC$, all $2$-simplices corresponding to invertible $2$-morphisms, and all simplices in dimension higher than $2$. The nerve of $2$-catego\-ries endowed with this marking, which we call \emph{natural} nerve and denote $\Nnat_2(\cC)$, is a variant of a marking considered by Harpas--Nuiten--Prasma \cite{HNP}, Lurie \cite{lurieGoodwillie} and Riehl \cite{EmilyNotes}. The natural nerve fits into an adjoint pair
$$\honat_2\colon\psh{t\Delta}\rightleftarrows 2\cat\colon\Nnat_2$$
that has a better homotopical behaviour.
Indeed, not only is the natural nerve is a right Quillen functor, but it in fact creates the homotopy theory of $2$-catego\-ries, in the sense of the following theorem, which will appear as \cref{Lackistransferred}.

\begin{unnumberedtheoremA}
Lack's model structure on $2\cat$ is right-transferred from the Riehl--Verity model structure on $\psh{t\Delta}$ for $2$-complicial sets along the natural nerve $\Nnat_2\colon2\cat\to\psh{t\Delta}$.
\end{unnumberedtheoremA}

Even though the natural nerve is not fully faithful on the nose, it is in a homotopical sense, as we will show in \cref{counitweakequivalence,corcofibrant}.

\begin{unnumberedtheoremB}
For any $2$-category $\cC$ the counit $\epsilon_{\cC}\colon\honat_2(\Nnat_2(\cC))\to\cC$ is a weak equivalence in the model structure for $(\infty,2)$-categories. As a consequence, the natural nerve is homotopically fully faithful and $\honat_2(\Nnat_2(\cC))$ is a cofibrant replacement for $\cC$ in the Lack model structure.
\end{unnumberedtheoremB}

The following relation between the natural nerve and the Roberts-Street nerve is given by \cref{naturalnervefibrantreplacement}.

\begin{unnumberedtheoremC}
For any $2$-category $\cC$ the natural nerve $\Nnat_2(\cC)$ is a fibrant replacement of the Roberts--Street nerve $\Ntdelta_2(\cC)$.
\end{unnumberedtheoremC}

\addtocontents{toc}{\protect\setcounter{tocdepth}{1}}
\subsection*{Acknowledgements}
We would like to thank Andrea Gagna, Lennart Meier, Francois M\'etayer and Emily Riehl for helpful conversations.

\section{Model structures for strict higher categories}
We recall a family of model structures on the categories $n\cat$ of strict higher categories for $n\ge0$. Although we describe several relevant constructions for general $n$, the reader is encouraged to focus on $n=2$, for which most contributions of this paper are. Although less surprising, the cases $n=0,1$ fit into the picture and are useful to familiarize the reader 
with the notions.

    The idea behind a strict higher category is that it consists of $m$-morphisms for positive $m$
    \begin{itemize}[leftmargin=*]
        \item whose source and target are two ($m-1$)-morphisms;
        \item that can be composed along $(m-1)$-morphisms;
        \item that have an identity $(m+1)$-morphism.
    \end{itemize}
    Composition is then required to be strictly unital and associative.
    
    If there are $m$-morphisms existing in every dimension $m\ge0$, this is formalized by the established notion of an \emph{$\omega$-category} \cite[\textsection1]{StreetOrientedSimplexes}. When there are only $m$-morphisms up to dimension $m\le n$, we recover the more familiar idea of a strict \emph{$n$-category}: assuming that a \emph{$0$-category} is just a set, an $n$-category is inductively described as a category enriched over the category of $(n-1)$-categories. All together, $\omega$-categories form a category $\omega\cat$ in which the morphisms are $\omega$-functors, and it contains as a full subcategory the category $n\cat$ of strict $n$-categories and strict $n$-functors for $n\ge0$.
    
    \begin{rmk}
\label{truncation}

The fully faithful 
inclusions
$i\colon n\cat\hookrightarrow(n+1)\cat\hookrightarrow\omega\cat$
fit into adjoint pairs
\[
\begin{tikzcd}[column sep=2cm]
 n\cat \arrow[r, ""{name=x1, above}, ""{name=x2, below, inner sep=1pt}, hook]& (n+1)\cat \arrow[l, bend left=25, "\lvert -\rvert_n"{name=x3, below, pos=0.56}] \arrow[l, bend right=25, "(-)_{n}"{name=x4, above}, ""{name=x5, below, inner sep=1pt, pos=0.56}]
 \arrow[from=x1, to=x3, symbol=\dashv]
 \arrow[from=x5, to=x2, symbol=\dashv]
\end{tikzcd}
\text{ and }
\begin{tikzcd}[column sep=2cm]
 n\cat \arrow[r, ""{name=x1, above}, ""{name=x2, below, inner sep=1pt}, hook]& \omega\cat \arrow[l, bend left=25, "\lvert -\rvert_n"{name=x3, below, pos=0.5}] \arrow[l, bend right=25, "(-)_{n}"{name=x4, above}, ""{name=x5, below, inner sep=1pt, pos=0.5}]
 \arrow[from=x1, to=x3, symbol=\dashv]
 \arrow[from=x5, to=x2, symbol=\dashv]
\end{tikzcd}
\]
where the left adjoint $(-)_n$ is the \emph{$n$-th truncation functor} and the right adjoint $|-|_n$ is the \emph{underlying $n$-category functor}.
Given an $\omega$-category or $(n+1)$-category $\cC$, the underlying $n$-category $|\cC|_n$ is obtained by forgetting all morphisms in dimension strictly higher than $n$ and the $n$-truncation $\cC_n$ of  can be understood as the $n$-category obtaining by forcing all morphisms in dimension strictly higher than $n$ to be identities. We will discuss an instance of truncation needed later more in details.
\end{rmk}

Lafont--M\'etayer--Worytkiewicz \cite{LMW} and Futia \cite{futia} identify a notion of $\omega$-weak equivalence 
that specializes to the meaningful notion of $n$-categorical equivalence. In particular, the notion recovers bijections of sets, equivalences of ordinary categories and biequivalences of $2$-categories when $n=0,1,2$. In order to model the homotopy theory of $\omega$-categories and of $n$-categories, Lafont--M\'etayer--Worytkiewicz in \cite{LMW} construct a model structure on $\omega\cat$ whose weak equivalences are the $\omega$-weak equivalences, and then right-transfer it along the inclusion functor to the category $n\cat$ of strict $n$-categories, obtaining a model structure with the correct weak equivalences.

 \begin{thm}[{\cite[Theorem 5]{LMW}}]
\label{modelstructureondiscretepresheaves}
The category $n\cat$ supports a cofibrantly generated model structure in which
\begin{itemize}[leftmargin=*]
\item all $n$-categories are fibrant;
\item the weak equivalences are precisely the $n$-categorical equivalences.
\end{itemize}
\end{thm}

When $n=0$ the model structure above specializes to the following well-known model structure on the category $\set$ of sets\footnote{For an account of model structures on $\set$, we recommend looking at \cite{CamarenaSets}.}.

\begin{prop}
\label{modelstructureSet}
The category $\set$ supports a cofibrantly generated model structure in which
\begin{itemize}[leftmargin=*]
\item the weak equivalences are the bijections;
 \item the fibrations are all functions;
\item the cofibrations are all functions.
\end{itemize}
Moreover,
\begin{itemize}[leftmargin=*]
 \item the unique generating acyclic cofibration is
 $$\id_{\{0\}}\colon\{0\}\to\{0\}$$
 \item the generating cofibration are
 $$\varnothing\hookrightarrow\{0\}\quad\text{ and }\quad\{0\}\amalg\{0\}\to\{0\}.$$
\end{itemize}
\end{prop}

When $n=1$ the model structure above specializes to the canonical model structure on the category $\cat$ of ordinary categories, attributed to Joyal--Tierney; see \cite{RezkCat} for a nice account. In order to recall the generating cofibrations and acyclic cofibrations, we introduce relevant notation.

\begin{notn}
We denote by $\mathbb I$ the \emph{free isomorphism}, namely the category containing two objects $x$ and $y$, and two non-identity $1$-morphisms $f\colon x\to y$ and $g\colon y\to x$ that are inverse to each other, namely $\id_x=gf$ and $fg=\id_y$. The terminology is due to the fact that for any category $\cC$, functors $\mathbb I\to\cC$ classify precisely isomorphisms in $\cC$.
\end{notn}

\begin{notn}
For $m\ge-1$, we denote by $[m]$ the totally ordered set (regarded as a category) with $m+1$ objects labelled $0,1,\dots,m$. In particular, $[-1]$ is the empty category and $[0]$ is the terminal category.
\end{notn}

\begin{notn}
We denote by $[\rightrightarrows]$ the \emph{free parallel $1$-morphism}, namely the $1$-category
freely generated by the graph
 \[
           \begin{tikzcd}
             x \arrow[r, bend left, ""{name=A, below}] \arrow[r, bend right,  ""{name=B, above}]& y
           \end{tikzcd}.
           \]
\end{notn}

\begin{thm}
The category $\cat$ supports a cofibrantly generated model structure, called the \emph{canonical model structure}, in which
\begin{itemize}[leftmargin=*]
\item the weak equivalences are the equivalences of categories;
\item the fibrations are the isofibrations, and in particular all categories are fibrant;
\item the cofibrations are the \emph{isocofibrations}, i.e, the functors that are injective on objects.
\end{itemize}
Moreover,
\begin{itemize}[leftmargin=*]
 \item the unique generating acyclic cofibration is
$$[0]\hookrightarrow \mathbb I,$$
 \item the generating cofibrations are
 $$[-1]\hookrightarrow[0],\quad[0]\amalg[0]\hookrightarrow[1]\quad\text{ and }\quad[\rightrightarrows]\to[1].$$
\end{itemize}
\end{thm}

When $n=2$ we obtain Lack's model structure on the category $2\cat$ of $2$-categories. In order to recall the generating cofibrations and acyclic cofibrations, we introduce relevant notation and terminology.

\begin{notn}
Given any category $\cC$ we denote by $\Sigma\cC$ the \emph{suspension} of $\cC$, namely the $2$-category given by two objects $x$ and $y$ and mapping categories as follows:
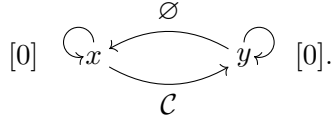
\begin{center}
    \begin{tikzpicture}
    \draw (0,0)  node[inner sep=0.2cm](a){} node(x){$x$};
    \draw (2,0) node[inner sep=0.2cm](b){} node(y){$y$};
    
     \draw[->] (a) edge[bend right] node[below](xy){$\cC$}(b);
     \draw[->] (b) edge[bend right] node[above](xy){$\varnothing$}(a);
     \draw[->] (a.90) arc (0:290:2mm)node[xshift=-0.8cm](xx){$[0]$};
     \draw[->] (b.90) arc (-180:-470:2mm)node[xshift=0.8cm](yy){$[0]$.};
    \end{tikzpicture}
\end{center}
\end{notn}

In particular, $\Sigma\mathbb I$ is the \emph{free $2$-isomorphism} and $\Sigma[\rightrightarrows]$ is the \emph{free parallel $2$-morphism}, namely the $2$-category freely generated by the $2$-graph
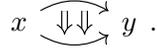
\[
\begin{tikzcd}
             x \arrow[r, bend left, ""{name=A, below, outer sep=0.05cm}] \arrow[r, bend right,  ""{name=B, above, outer sep=0.05cm}]& y
             \ar[from=A.east, to=B.east, Rightarrow]
             \ar[from=A.west, to=B.west, Rightarrow]
           \end{tikzcd}.
           \]

We recall the following standard $2$-categorical terminology relative to any $2$-category $\cC$; for details see e.g.~\cite{LackCompanion}.
\begin{itemize}[leftmargin=*]
    \item An \emph{adjunction} in $\cC$ consists of a pair of $1$-morphisms $f\colon x\to y$  and $g\colon y\to x$, called left adjoint and right adjoint, together with a pair of $2$-morphisms $\eta\colon\id_{x}\Rightarrow gf$ and $\epsilon\colon fg\Rightarrow\id_y$, called unit and counit, satisfying the triangle identities
    $$(\epsilon*f)\circ(f*\eta)=\id_f\text{ and }(g*\epsilon)\circ(\eta*g)=\id_g.$$
    \item An \emph{adjoint equivalence} in $\cC$ is an adjunction for which the unit and the counit are invertible; in this case, as a consequence of the triangle identities the unit and the counit satisfy the identities
     $$f*\eta=(\epsilon*f)^{-1}\colon f\Rightarrow fgf\text{ and }g*\epsilon=(\eta*g)^{-1}\colon gfg\Rightarrow g.$$  
    \item An \emph{equivalence} in $\cC$ is a $1$-morphism $f\colon x\to y$ for which there exists a $1$-morphism $g\colon y\to x$ such that there are $2$-isomorphisms
     $$\id_x\cong gf\colon x\to x\text{ and }fg\cong\id_y\colon y\to y,$$
    or equivalently (cf.~\cite[2.2]{LackCompanion}) 
    a $1$-morphism that fits into an adjoint equivalence as either a left adjoint or a right adjoint.
\end{itemize}

\begin{notn}
We denote by $\mathbb E$\footnote{This $2$-category is $\textrm{Adj}$ following the terminology of \cite{HNP}.} the \emph{free adjoint equivalence}, namely the $2$-category freely generated by two objects $x$ and $y$, two non-identity $1$-morphisms $f\colon x\to y$  and $g\colon y\to x$, and two non-identity $2$-morphisms $\eta\colon\id_{x}\Rightarrow gf$ and $\epsilon\colon fg\Rightarrow\id_y$ satisfying the relations
 $$f*\eta=(\epsilon*f)^{-1}\colon f\Rightarrow fgf\text{ and }g*\epsilon=(\eta*g)^{-1}\colon gfg\Rightarrow g.$$
The terminology is due to the fact that for any $2$-category $\cC$, $2$-functors $\mathbb E\to\cC$ classify precisely adjoint equivalences in $\cC$.
\end{notn}

 \begin{thm}[\cite{lack1}]
\label{modelstructure2cat}
The category $2\cat$ supports a cofibrantly generated model structure in which
\begin{itemize}[leftmargin=*]
\item the weak equivalences are the \emph{biequivalences}, i.e., biessentially surjective functors that are homwise equivalences of categories;
\item all $2$-categories are fibrant;
\item the cofibrant objects are the $2$-categories whose underlying graph is free.
\end{itemize}
Moreover,
\begin{itemize}
\item the generating acyclic cofibrations are the canonical inclusions\footnote{We warn the reader that in the paper \cite{lack1} Lack identified as a generating acyclic cofibration the inclusion of $[0]$ into the free equivalence instead of the free adjoint equivalence. It was pointed out by Joyal that the original choice did not implement the desired properties (in particular the former inclusion is not a weak equivalence), and the mistake was then corrected in \cite{lack2}.}
$$[0]\hookrightarrow \mathbb E\quad\text{and}\quad[1]=\Sigma[0]\hookrightarrow \Sigma{\mathbb I};$$
\item the generating cofibrations are the canonical maps
$$[-1]\hookrightarrow[0],\ [0]\amalg[0]\hookrightarrow[1],\ \Sigma([0]\amalg[0])\hookrightarrow \Sigma{[1]}\ \text{and}\  \Sigma[\rightrightarrows]\to\Sigma{[1]}.$$
\end{itemize}
\end{thm}

Since all model structures on the categories of $n$-categories are right-transferred from the model structure on the category of $\omega$-categories, as $n$-varies these model structure are compatible in the following sense.

\begin{prop}
For any $n\ge0$, the model structure on $n\cat$ is right-transferred from the model structure $(n+1)\cat$ along the canonical inclusion $i\colon n\cat\to(n+1)\cat$.
In particular, the adjunction
$$(-)_n\colon(n+1)\cat\rightleftarrows n\cat\colon i$$
is a Quillen pair.
\end{prop}

\section{Model structures for weak higher categories}
We recall a family of model structures defined on a presheaf category $\psh{t\Delta}$ whose fibrant objects are a model of $(\infty,n)$-categories for $n\ge0$.
This variant of Riehl--Verity's model of $(\infty,n)$-categories is presented by a model category of presheaves over the category $t\Delta$ satisfying further fibrancy conditions.
For the reader's convenience, we recall the relevant definitions; for a more detailed account on the content of this subsection, see \cite{or}.

We start by giving a description of Verity's category $t\Delta$ from \cite{VerityComplicialI}, in terms of its generators and relations.

\begin{notn}
\label{tDelta}
Let $t\Delta$ be the category defined as follows.
The set of objects is given by
$$\Ob(t\Delta):=\{[m]\text{ for all }m\geq 0\}\cup\{[m]_t\text{ for all }m\geq 1\}.$$
The morphisms in $t\Delta$ are generated by the following maps:
\begin{itemize}
\item cofaces $d^i\colon[m]\to[m+1]$ for $0\leq i \leq m+1$,
\item codegeneracies $s^i\colon[m]\to[m-1]$ for $0 \leq i \leq m-1$,
\item counmarking maps $\varphi\colon[m]\to[m]_t$ for $m\geq 1$,
\item comarking maps
$\zeta^{i}_m\colon[m]_t\to[m-1]$ for $0\leq i \leq m-1$. 
\end{itemize}
subject to the usual cosimplicial identities for cofaces and codegeneracies 
and the additional relations 
\begin{alignat*}{2}
\zeta^i\varphi &=s^i\colon [m+1]\to[m]\text{ for }m\ge1\text{ and }0\leq i \leq m,\\
s^i\zeta^{j+1} &=s^j\zeta^{i} \colon [m+2]_t \to [m]\text{ for }0\leq i \leq j \leq m.
\end{alignat*}
The generating morphisms of $t\Delta$ can be pictured as follows.
\[
\begin{tikzcd}[row sep=large, column sep=large]
{[0]}
 \arrow[r, arrow, shift left=1ex] \arrow[r, arrow, shift right=1ex] 
& 
{[1]}\arrow[l, arrow]  \arrow[r, arrow] \arrow[r, arrow, shift left=1.5ex] \arrow[r, arrow, shift right=1.5ex]
\arrow[d, arrow, "\varphi"]& 
{[2]}
\arrow[l, arrowshorter, shift left=0.75ex] \arrow[l, arrowshorter, shift right=0.75ex] 
\arrow[d, arrow, "\varphi"]&[-0.8cm]\cdots\\
 & 
{[1]_t}  
\arrow[ul, arrowshorter, "\zeta^0"]
& 
{[2]_t} \arrow[ul, arrowshorter, shift left=0.5ex, "\zeta^0"] \arrow[ul, arrowshorter, shift right=0.5ex,"\zeta^1" swap] 
&\cdots
\end{tikzcd}
\]
\end{notn}

\begin{defn}
A \emph{prestratified simplicial set}, or \emph{$t\Delta$-set} for short, is a pre\-sheaf $X\colon t\Delta^{\op}\to\set$. 
A \emph{stratified simplicial set} is a prestratified simplicial set for which the structure maps $X([m]_t)\to X([m])$ are injective.
\end{defn}

If $X$ is a stratified simplicial set, we can think of $X([m]_t)$ as a specified subset of the set $X_m$ of $m$-simplices. Equivalently, the data of a stratified simplicial set consists of a simplicial set together with a collection of marked simplices in dimension at least $1$ containing all degenerate ones. Generalizing the same point of view, we will regard an arbitrary $t\Delta$-set $X$ as a simplicial set in which certain simplices are marked possibly many times, and $X([m]_t)$ is the set of labels for marked $m$-simplices.

Unless otherwise specified, given any simplicial set we read it as a stratified simplicial set in which only the degenerate simplices are marked, each uniquely.

In order to describe the model structure on $\psh{t\Delta}$ for $(\infty,n)$-categories, it is necessary to introduce further notational conventions.\footnote{Due to the involved combinatorial nature of these objects, we encourage the reader who is not familiar with them to have a look at e.g.~ \cite{EmilyNotes} where they are introduced together with insightful discussions.}

\begin{notn} 
\label{preliminarynotation}
We denote
\begin{itemize}[leftmargin=*]
    \item by $\Delta[m]_{(t)}$ the $t\Delta$-set represented by $[m]_{(t)}$, which is in fact a stratified simplicial set.
    \item by $\Delta^k[m]$, for $0\leq k \leq m$, the stratified simplicial set whose underlying simplicial set is $\Delta[m]$ and in which a non-degenerate simplex is marked if and only if it contains the vertices $\{k-1,k,k+1\}\cap [m]$.
    \item by $\Delta^k[m]'$, for $0\leq k \leq m$, the stratified simplicial set obtained from $\Delta^k[m]$ by additionally marking the $(k-1)$-st and $(k+1)$-st face of $\Delta[m]$.
    \item by $\Delta^k[m]''$, for $0\leq k \leq m$, the stratified simplicial set obtained from $\Delta^k[m]'$ by additionally marking the $k$-th face of $\Delta[m]$.
    \item by $\Lambda^k[m]$, for $0\leq k \leq m$, the stratified simplicial set whose underlying simplicial set is the $k$-horn $\Lambda^k[m]$ and whose simplex is marked if and only if it is marked in $\Delta^k[m]$. 
    \item by $\eqDelta$ the stratified simplicial set whose underlying simplicial set is $\Delta[3]$, and the non-degenerate marked simplices consist of all $2$- and $3$-simplices, as well as $1$-simplices $[02]$ and $[13]$.
        \item by $\Delta[3]^\sharp$ the stratified simplicial set whose underlying simplicial set is $\Delta[3]$, and all simplices in positive dimensions are marked.
\end{itemize}
\end{notn}

\begin{defn}
\label{anodynemaps}
An \emph{elementary anodyne extension} is one of the following.
 \begin{enumerate}[leftmargin=*]
  \item  The \emph{complicial horn extension}, i.e., the canonical map
  i.e., the regular inclusion
 $$\Lambda^k[m]\to \Delta^k[m]\text{ for $m\geq 1$ and $0\leq k\leq m$},$$
 which is an ordinary horn inclusion on the underlying simplicial sets.
 \item The \emph{complicial thinness extension}, i.e., the canonical map 
$$\Delta^k[m]' \to \Delta^k[m]''\text{ for $m\geq 2$ and $0\leq k \leq m$},$$
which is an identity on the underlying simplicial set.
\item The \emph{triviality extension} map,
i.e., the canonical map
$$\Delta[l]\to \Delta[l]_t\text{ for $l>n$},$$
which is an identity on the underlying simplicial set.
\item The \emph{saturation extension}, i.e., the canonical map
$$\Delta[l]\star \Delta[3]_{eq}  \to \Delta[l]\star \Delta[3]^{\sharp}\text{ for $l\geq -1$},$$
which is an identity on the underlying simplicial set.
Here, $\star$ denotes the join construction of stratified simplicial sets, which can be found e.g.~in \cite[Observation 34]{VerityComplicialI} or \cite[Def. 3.2.5]{EmilyNotes}.
\end{enumerate}
\end{defn}

Roughly speaking, according to the intuition that $n$-simplices in a $t\Delta$-set or stratified simplicial sets are $n$-morphisms, and that the marked simplices are $n$-morphisms that are invertible in some sense, we can interpret the right-lifting properties with respect to the four classes of elementary anodyne extensions as follows.
\begin{enumerate}[leftmargin=*]
    \item Right-lifting against the complicial horn anodyne extensions encodes the fact that morphisms can be composed.
    \item Right-lifting against the thinness anodyne extensions encodes the fact that composites of equivalences are also equivalences. 
    \item Right-lifting against the $l$-trivial anodyne extensions encodes the fact that morphisms in dimension higher than $l$ are invertible.
    \item Right-lifting against the saturation anodyne extensions encodes the fact that all equivalences admit an inverse in some sense. 
\end{enumerate}

We will be interested in the following objects, which will be the fibrant objects of the desired model structure.

\begin{defn}
An \emph{$(\infty,n)$-category}\footnote{We warn the reader that Verity \cite{VerityComplicialAMS} uses the terminology ``pre-complicial'' to mean something different.
}
is an $n$-precomplicial set, i.e., a $t\Delta$-set that has the right lifting property with respect to the elementary anodyne extensions from \cref{anodynemaps}.
\end{defn}

The terminology is justified by the fact that $n$-precomplicial sets are equivalent to $n$-complicial sets, which are a proposed model of $(\infty,n)$-categories. For the convenience of the reader who is interested in comparing with the existing literature, we include a small dictionary on the terminology appearing in other sources.

\begin{digression}
Historically, the elementary anodyne extensions from \cref{anodynemaps} were introduced to define in terms of right lifting properties several relevant types of $t\Delta$-sets or stratified simplicial sets.
\begin{itemize}[leftmargin=*]
    \item A \emph{strict complicial set}, is a stratified simplicial set that has the right lifting property with respect to the complicial horn and the thinness anodyne extensions and such that the lifts are unique. Strict complicial sets were first considered by Roberts in unpublished work and by Street in \cite[\textsection 5]{StreetOrientedSimplexes}, 
    and Verity showed in \cite{VerityComplicialAMS} that strict complicial sets correspond to $\omega$-categories via a suitable nerve construction (which will be revised as \cref{Veritynerve}).
    As a natural weakening of the previous notion, a \emph{weak complicial set} is a stratified simplicial set that has the right lifting property with respect to the complicial horn and the thinness anodyne extensions. The theory of weak complicial sets is anticipated by Street in \cite{StreetOrientedSimplexes} and is then developed by Verity in \cite{VerityComplicialI}. 
    \item An \emph{$n$-trivial stratified simplicial set} is a stratified simplicial set that has the right lifting property with respect to the $l$-th triviality anodyne extension for $l>n$.
    \item A \emph{saturated weak complicial set} is a weak complicial set that has the right lifting property with respect to the saturation anodyne extensions; the notion of saturation first appears in \cite{EmilyNotes}.
    \item An \emph{$n$-complicial} set is a stratified simplicial set that has the right lifting property with respect to the elementary anodyne extensions; $n$-complicial sets are first consider in \cite{EmilyNotes} as a candidate model for $(\infty,n)$-categories and the terminology is introduced later in \cite{RiehlVerityBook}. As a variant, an \emph{$n$-precomplicial set} is a $t\Delta$-set that has the right lifting property with respect to the elementary anodyne extensions. We introduced $n$-complicial sets in \cite{or}, where we also show that there are two Quillen equivalent model structures on the categories of stratified simplicial sets and $t\Delta$-sets; see \cite[Prop.1.31]{or} for more details.
\end{itemize}
\end{digression}

We now recall the model structure for $(\infty,n)$-categories.

\begin{thm}[{\cite[Theorem 1.28]{or}}]
\label{modelstructureondiscretepresheaves}
The category $\pstrat$ supports a cofibrantly generated model structure where
\begin{itemize}[leftmargin=*]
\item the fibrant objects are precisely the $(\infty,n)$-categories;
\item the cofibrations are precisely the monomorphisms.
\end{itemize}
Moreover,
\begin{itemize}[leftmargin=*]
\item the generating cofibrations are the canonical inclusions
$$\partial\Delta[m]\to\Delta[m]\text{ for $m\ge0$ and }\Delta[m]\to\Delta[m]_t\text{ for }m\ge1.$$
\end{itemize}
\end{thm}

\begin{prop}
For any $n\ge0$, the identity
$$\id\colon\psh{t\Delta}_{(\infty,n+1)}\rightleftarrows\psh{t\Delta}_{(\infty,n)}\colon \id$$
defines a Quillen pair between the model structure on $\psh{t\Delta}$ for $(\infty,n)$-categories and for $(\infty,n+1)$-categories. In particular any $(\infty,n)$-category is also an $(\infty,n+1)$-category.
\end{prop}

We want to focus on $n=0,1,2$ and relate the (model) categories of strict $n$-categories and of $(\infty,n)$-categories by means of meaningful adjunctions.

\section{The Roberts--Street nerve-categorification adjunction}
\label{naivenervesection}

In this section we consider an adjunction between $n\cat$ and $\psh{t\Delta}$, which will then be refined to one with better homotopical properties in the next section. While we present the construction for general $n\ge0$, we then explain which shapes these constructions take for $n=0,1,2$.

The key ingredient to define the desired adjunction is the following.

\begin{notn}
We denote by $\mathcal O[m]$ the \emph{$m$-th oriental}, as defined by Street in \cite[\textsection2]{StreetOrientedSimplexes}; see also \cite[\textsection2.2]{EmilyNotes} for an alternative account.
\end{notn}

While the precise definition of orientals is quite involved, the reader can keep in mind that $\mathcal O[m]$ is the free $m$-category generated by the standard simplex $\Delta[m]$. For instance, it contains $m+1$ objects, and a top $m$-morphism between two suitable pasted composites of $(m-1)$-morphisms. It is straightforward to see that orientals form a cosimplicial object in $\omega\cat$.

In low degrees, we find that the $0$-th oriental is the singleton $\mathcal O[0]\cong\{0\}$, the $1$-st oriental is the walking arrow category $\mathcal O[1]\cong[1]$, the $2$-nd oriental $\mathcal O[2]$ is the $2$-category depicted as
\[\begin{tikzcd}[baseline=(current  bounding  box.center)]
 & 1 \arrow[rd, ""]  & \\
    0 \arrow[ru, "{}"]
  \arrow[rr, ""{below}, ""{name=D,inner sep=1pt}]
  && 2,
  \arrow[Rightarrow, from=D, 
 to=1-2, shorten >= 0.1cm, shorten <= 0.1cm, ""]
\end{tikzcd}
\]
    and the $3$-rd oriental $\mathcal O[3]$ is the $3$-category that can be depicted as
    \[\simpftri{}{}{}{}{}{}{}{}{}\simpftricontinued{}.\]

Orientals were used by Street in \cite[\textsection5]{StreetOrientedSimplexes} to construct the following adjunction, which was historically the first attempt to produce an object of a simplicial nature starting from a higher category.

 \begin{const}
 \label{Veritynerve}
 The cosimplicial object in $\omega\cat$
 given by $[m]\mapsto\mathcal O[m]$
 induces an adjunction
 $$\hostreet\colon\psh{\Delta}\rightleftarrows\omega\cat\colon \Nstreet.$$
 We call $\Nstreet$ the \emph{Street nerve} and $\hostreet$ the \emph{Street categorification}.
 \end{const}

In particular, for any $\omega$-category $\cC$ the set of $m$-simplices of the Street nerve is given by
 $$\Nstreet(\cC)_m\cong\Hom_{\omega\cat}(\mathcal O[m],\cC).$$

By composing with the truncation adjunction from \cref{truncation} we obtain another relevant adjunction. The cosimplicial object that defines it is the following.

\begin{notn}
For $m\ge0$, we denote by $\cO_n[m]$ the \emph{$n$-truncated $m$-th oriental} $\mathcal O[m]$, namely the $n$-truncation of the $m$-th oriental.
\end{notn}
In low degrees, we find that $\mathcal O_0[1]$ is a singleton, $\mathcal O_1[2]\cong[2]$ is the category with two composable morphisms, $\mathcal O_2[3]$ the $2$-category depicted as
$$
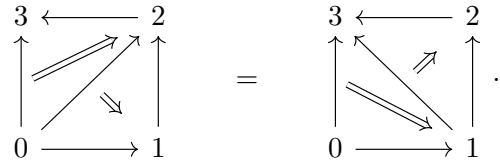
\simpf{}{}{}{}{}{}{}{}{}\simpfcontinued{}.
$$

The following adjunction was also considered in \cite[\textsection 5.10]{AraMaltsiniotisVers}.
The right adjoint is the ordinary nerve for $n=1$ and was first considered by Duskin \cite{duskin} in the case $n=2$.
 \begin{const}
 The cosimplicial object in $n\cat$ 
 given by $[m]\mapsto\mathcal O_n[m]$
 induces an adjunction
 $$\hostreet_n\colon\psh{\Delta}\rightleftarrows n\cat\colon\Nstreet_n.$$
 We call $\Ntdelta_n$ the \emph{$n$-Street nerve} and $\hotdelta_n$ the \emph{$n$-Street categorification}.
 \end{const}

In particular, for any $\omega$-category $\cC$ the set of $m$-simplices of the Street nerve is given by
 $$\Nstreet_n(\cC)_m\cong \Hom_{\omega\cat}(\mathcal O[m],\cC).$$

 When working with $n>1$, simplices in dimension $m$ in the nerve of an $n$-category $\cC$ play the twofold role of witnessing $m$-morphisms of $\cC$, as well as witnessing the composite of morphisms in lower dimension of $\cC$.
 It then became apparent that it is useful to record which simplices in the Street nerve of an $n$-category $\cC$ were witnessed by morphisms of $\cC$ that were to be considered equivalences in some sense. The first attempt in this direction is due to Street and Roberts, who considered the Street nerve of an $n$-category $\cC$ and marked the simplices witnessed by identity morphisms of $\cC$. The resulting nerve is the right adjoint in the following adjunction.

 \begin{const}
Consider the co-$t\Delta$-object in $n\cat$
 given on objects by 
 $$[m]\mapsto\mathcal O_n[m]\quad\text{ and }\quad[m]_t\mapsto\left\{
 \begin{array}{ccc}
\cO_{m-1}[m]&m\le n+1\\
 \mathcal O_n[m]&m\ge n+1
 \end{array}
 \right.$$
 with cosimplicial structure induced by that of the truncated oriental $\mathcal O_n[m]$, and with further structure maps given by taking the $(m-1)$-truncation
 $\mathcal{O}_n[m]\to (\mathcal{O}_n[m])_{m-1}\cong\mathcal O_{m-1}[m]$
 and by factoring the codegeneracy maps $\mathcal O_n[m]\to\mathcal O_n[m-1]$ through $\mathcal O_{m-1}[m]$ or $\mathcal O_n[m]$.

This co-$t\Delta$-object
 induces an adjunction
 $$\hotdelta_n\colon\psh{t\Delta}\rightleftarrows n\cat\colon\Ntdelta_n.$$
 We call $\Ntdelta$ the \emph{Roberts-Street $n$-nerve} and $\hotdelta$ the \emph{Roberts-Street $n$-cate\-gori\-fication}.
 \end{const}

We see that the set of $m$-simplices of $\Ntdelta_n(\cC)$ is given by
   $$\Ntdelta_n(\cC)_m\cong\Hom_{n\cat}(\mathcal O_n[m],\cC)\cong \Hom_{\omega\cat}(\mathcal O[m],\cC).$$
 Therefore $\Ntdelta_n(\cC)$ is a stratified simplicial set whose underlying simplicial set is its Street nerve $\Nstreet_n(\cC)$,
 and the marked $m$-simplices are those witnessed by an identity $m$-morphism in $\cC$, so in particular and all $m$-simplices for $m>n$.

The main property of the nerve $\Ntdelta_n$ is the following, which is a reformulation of a theorem that was first conjectured by Roberts and Street \cite[\textsection 5]{StreetOrientedSimplexes} 
and later shown by Verity \cite{VerityComplicialAMS}.

\begin{thm}
The nerve $\Ntdelta_n\colon n\cat\to\psh{t\Delta}$ is a fully faithful functor
whose essential image is given by the stratified simplicial sets that are strict $n$-complicial sets.
 \end{thm}
 
The property of fully faithfulness is equivalently expressed as follows.
 
\begin{prop}
\label{naivecounitiso}
For any $n$-category $\cC$ the counit of the adjunction above yields an isomorphism of $n$-categories
  \[
    \begin{tikzcd}
     \epsilon^{RS}_{\cC}\colon\hotdelta_n(\Ntdelta_n(\cC))\arrow[r, "\cong"]&\cC.
    \end{tikzcd}
    \]
 \end{prop}

In order to get the reader familiar with the construction, we will now specialize these constructions in the cases $n=0,1,2$.

For $n=0$, the $0$-truncated orientals are singletons, and the nerve-categori\-fication adjunction specializes to the following.

 \begin{const}
 The co-$t\Delta$-object in $\set$
 given by $[m]_{(t)}\mapsto\{0\}$
 induces an adjunction
 $$\hotdelta_0\colon\psh{t\Delta}\rightleftarrows\set\colon\Ntdelta_0.$$
 \end{const}
In particular, $\Ntdelta_0(X)$ reads a set $X$ as a discrete $t\Delta$-set, or equivalently as a simplicial set constant at $X$ in which all simplices are marked, each uniquely.

 For $n=1$, the $1$-truncated $m$-th oriental is the poset $[m]$, and the nerve-categorification adjunction specializes to the following.
 
 \begin{const}
 The co-$t\Delta$-object in $\cat$
 given by
 $$[m]\mapsto[m]\quad\text{ and }\quad[m]_t\mapsto\left\{
 \begin{array}{ccc}
 {[0]}&m=1\\
{[m]}&m\ge1
 \end{array}
 \right.$$
 induces an adjunction
 $$\hotdelta_1\colon\psh{t\Delta}\rightleftarrows\cat\colon\Ntdelta_1.$$
 \end{const}
 
In particular, $\Ntdelta_1(\cC)$ is a stratified simplicial set whose underlying simplicial set is the ordinary nerve $N(\cC)$,
 the marked $1$-simplices are the degenerate ones, and all $m$-simplices for $m>1$.

For $n=2$, the $2$-truncated orientals $\mathcal O_2[m]$ seen as simplicial categories by applying the nerve homwise coincide with the simplicial categories
$$N_*(\mathcal O_2[m])\cong\mathfrak C[\Delta[m]],$$
which were introduced by Cordier--Porter \cite{CordierPorterHomotopyCoherent} and later used by Joyal \cite{Joyal2007} and Lurie \cite[Def.1.1.5.1]{htt} to define the homotopy coherent nerve adjunction.

In this case, the nerve-categorification adjunction specializes to the following.

 \begin{const}
The co-$t\Delta$-object in $\twocat$
 given by
 $$[m]\mapsto\mathcal O_2[m]\quad\text{ and }\quad[m]_t\mapsto\left\{
 \begin{array}{ccc}
 [0]&m=1\\
 {[2]}&m=2\\
 \mathcal O_2[m]&m\ge 3
 \end{array}
 \right.$$
 induces an adjunction
 $$\hotdelta_2\colon\psh{t\Delta}\rightleftarrows2\cat\colon\Ntdelta_2.$$
 \end{const}

 In particular, $\Ntdelta_2(\cC)$ is a stratified simplicial set whose underlying simplicial set is its Street--Duskin nerve  $\Nstreet_2(\cC)$,
 the marked $1$-simplices are the degenerate ones, the marked $2$-simplices are those that witness commutative triangles in $\cC$ and all $m$-simplices are marked in dimension $m>2$.

For the reader's convenience, we recall that the Street--Duskin nerve \linebreak $\Nstreet_2(\cC)$ of a $2$-category $\cC$ is a $3$-coskeletal simplicial set in which
\begin{enumerate}[leftmargin=*]
\item[(0)] a $0$-simplex consists of an object of $\cC$
$$x;$$
    \item a $1$-simplex consists of a $1$-morphism of $\cC$
     \[
\begin{tikzcd}
    x \arrow[rr, "a"{below}]&& y;
\end{tikzcd}
\]
    \item a $2$-simplex consists of a $2$-morphism of $\cC$ of the form
    $$\begin{tikzcd}[baseline=(current  bounding  box.center)]
 & y \arrow[rd, "b"]  & \\
    x \arrow[ru, "{a}"]
  \arrow[rr, "c"{below}, ""{name=D,inner sep=1pt}]
  && z;
  \arrow[Rightarrow, from=D, 
 to=1-2, shorten >= 0.1cm, shorten <= 0.1cm, ""]
\end{tikzcd}$$
\item a $3$-simplex consists of four $2$-morphisms of $\cC$ that satisfy the following relation.
\[
\simpfver{d}{c}{e}{a}{b}{f}{}{}{}\simpfvercontinued{}{x}{y}{z}{w}
\]
\end{enumerate}

 \section{The natural nerve-categorification adjunction}
 \label{naturalnervesection}
 
For the easy case of $n=0$, the nerve-categorification adjunction introduced in the previous section is well-behaved homotopically, as recorded by the following proposition.
 
    \begin{prop}
    \label{modelstructureonSettransferred}
   The model structure on $\set$ from \cref{modelstructureSet} is right-transferred from the Riehl--Verity model structure on $\psh{t\Delta}$ for $(\infty,0)$-categories along the natural nerve $\Ntdelta_0\colon\set\to\psh{t\Delta}$. In particular, the adjunction
   $$\hotdelta_0\colon\psh{t\Delta}_{(\infty,0)}\leftrightarrows\set\colon\Ntdelta_0$$
   is a Quillen pair.
 \end{prop}
 
 \begin{proof}
 The proof is a straightforward exercise, after observing that in the model structure for $(\infty,0)$-categories a map between discrete $t\Delta$-sets is always a fibration, and is a weak equivalence if and only if it is an isomorphisms.
 \end{proof}
 
 This property of the nerve-categorification adjunction highly exploits the degenerate nature of the case $n=0$, and the rest of the section is devoted to the explaining why it fails for higher $n$ and to producing a different nerve-categorification adjunction for which the issues are resolved.
 
  \addtocontents{toc}{\protect\setcounter{tocdepth}{1}}
\subsection{The natural nerve-categorification adjunction for $n=1$}
 
The argument presented above cannot be reproduced at no price even for $n=1$, as there are two evidents obstructions.

\begin{rmk}
The adjunction
   $\hotdelta_1\colon\psh{t\Delta}\leftrightarrows\cat\colon\Ntdelta_1$
   is not a Quillen pair. Indeed, it will be shown in \cref{naturalnervefibrantreplacement} (or it can be checked directly) that $\Ntdelta_1(\cC)$ is fibrant if and only if $\cC$ does not contain non-identity isomorphisms.
\end{rmk}

  \begin{rmk}
  The model structure for $1$-complicial sets on $\psh{t\Delta}$ cannot be transferred to $\cat$ along the nerve $\Ntdelta_1\colon\cat\to\psh{t\Delta}$. If the said transferred model structure existed,
  the unique map $\mathbb I\to[0]$ would be an acyclic cofibration, being the image of the anodyne extension $\eqDelta\to\Delta[3]^\sharp$ via $\hotdelta_1$. Then the unique  map $B\mathbb Z\to[0]$, which can be expressed as a pushout of $\mathbb I\to[0]$ along the map $\mathbb I\to B\mathbb Z$ whose image is the generator of $\mathbb Z$
  would have to be an acyclic cofibration and in particular a weak equivalence in the transferred model structure. This would mean that $\Ntdelta_1(B\mathbb Z)\to\Ntdelta_1([0])=\Delta[0]$ is a weak equivalence, and in particular Kan weak equivalence of underlying simplicial sets, which is not true.
 \end{rmk}

Aiming to prove an analog of \cref{modelstructureonSettransferred} for $n=1$, we modify the nerve-categorification adjunction so that it has better homotopical properties. Inspired by Lurie's natural marking of quasi-categories \cite[Def. 3.1.1.8]{htt}, a way to solve the issues is to change the stratification of the Roberts--Street nerve of a $1$-category $\cC$ by marking the $1$-simplices of $\Ntdelta(\cC)$ witnessed by morphisms of $\cC$ that are isomorphisms, rather than identities. This stratification for the nerve of $1$-categories was first considered in \cite[Prop.3.1.8]{EmilyNotes}, and is referred to as the \emph{$1$-trivial saturated stratification}.
To implement this idea, recall that $\mathbb I$ denotes the free isomorphism.

  \begin{const}
 Consider the co-$t\Delta$-object in $\cat$
 given on objects by
 $$[m]\mapsto\mathcal [m]\quad\text{ and }\quad[m]_t\mapsto\left\{
 \begin{array}{cc}
{\mathbb I}&m=1\\
{ [m]}&m\ge2
 \end{array}
 \right.$$
  with the usual cosimplicial structure, and with the further structure map given by the canonical inclusion $[1]\to \mathbb{I}$.
It induces an adjunction
 $$\honat_1\colon\psh{t\Delta}\rightleftarrows\cat\colon\Nnat_1.$$
  We call $\Nnat_1$ the \emph{natural $1$-nerve} and $\honat_1$ the \emph{natural $1$-categorification}.
 \end{const}
 
  In particular, $\Nnat_1(\cC)$ is a stratified simplicial set whose underlying simplicial set is the ordinary nerve $N(\cC)$, the marked $1$-simplices are those corresponding to isomorphisms and all $m$-simplices are marked for $m>1$.
 
This adjunction has the desired homotopical properties.
   \begin{thm}
   \label{modelstructuretransferred1}
    The canonical model structure on $\cat$ is right-transferred from the Riehl--Verity model structure on $\psh{t\Delta}$ for $(\infty,1)$-categories along the natural nerve $\Nnat_1\colon\cat\to\psh{t\Delta}$.
    In particular, the adjunction
   $$\honat_1\colon\psh{t\Delta}_{(\infty,1)}\leftrightarrows\cat\colon\Nnat_1$$
   is a Quillen pair.
 \end{thm}
 
 This proposition can be checked directly using the formalism of Lurie's marked simplicial sets from \cite{htt}, or it can be deduced from the higher analog of the same result, which we will show as \cref{Lackistransferred}.

The natural nerve for $n=1$ is still fully faithful.
 
 \begin{prop}
 \label{counitiso1}
For any category $\cC$ the (derived) counit
    \[
    \begin{tikzcd}
     \epsilon^{\natural}_{\cC}\colon\honat_1(\Nnat_1(\cC))\arrow[r, two heads, "\cong"]&\cC
    \end{tikzcd}
    \]
    is an isomorphism of categories.
 \end{prop}

 \begin{proof}
We give two strategies to prove the proposition, and leave the details to the interested reader.
\begin{enumerate}[leftmargin=*]
    \item The counit can be seen to be an acyclic fibration, by showing that it has the right lifting property with respect to all generating cofibrations, so it is in particular an equivalence of categories. Moreover it is bijective on objects, and any equivalence that is bijective on objects is an isomorphism.
    \item By definition, the category $\honat_1(\Nnat_1(\cC))$ is obtained by glueing a copy of $[n]$ for any functor $[n]\to\cC$, and a copy of $\mathbb I$ for any invertible morphism $[1]\to\cC$. This means that the construction builds first $\cC$, and then glues an inverse to any morphism $f$ of $\cC$ that is already invertible. In particular, the newly added inverse of $f$ must agree with the already present inverse $f^{-1}$. In particular, $\honat_1(\Nnat_1(\cC))$ is isomorphic to $\cC$ and the counit realizes this isomorphism.
    \qedhere
\end{enumerate}
 \end{proof}

  \addtocontents{toc}{\protect\setcounter{tocdepth}{1}}
\subsection{The natural nerve-categorification adjunction for $n=2$}

We now proceed to studying the case of $n=2$. Inspired by the same philosophy, we want to change the stratification of the Roberts--Street nerve of a $2$-category $\cC$ by marking the $1$-simplices of $\Ntdelta_2(\cC)$ witnessed by $1$-morphisms of $\cC$ that are equivalences, rather than identities or isomorphisms, and by marking the $2$-simplices of $\Ntdelta_2(\cC)$ witnessed by $2$-morphisms of $\cC$ that are isomorphisms, rather than identities.

We implement this idea, by considering a multiply marked variant of this stratification for the nerve of $2$-categories, closely related to Riehl's  \emph{$2$-trivial saturated stratification} from \cite[Prop.3.1.10.]{EmilyNotes}.

Recall that $\mathbb E$ denotes the free adjoint equivalence, and that $\Sigma\mathbb I$ denote the suspension of the free isomorphism, which coincides with the free $2$-isomorphism.

\begin{notn}
 We denote by $\mathbb O_2[2]$ the $2$-category obtained from $\mathcal{O}[2]$ by adding a strict inverse to the top $2$-morphism; it can be written as a pushout
 $$\mathbb O[2]\cong\mathcal O_2[2]\aamalg{\Sigma[1]}\Sigma{\mathbb I},$$
 and can be pictured as follows
     $$\begin{tikzcd}[baseline=(current  bounding  box.center)]
 & 1 \arrow[rd, ""]  & \\
    0 \arrow[ru, ""]
  \arrow[rr, ""{below}, ""{name=D,inner sep=1pt}]
  && 2.
  \arrow[Rightarrow, from=D, 
 to=1-2, shorten >= 0.1cm, shorten <= 0.1cm, "\cong"]
\end{tikzcd}$$
 \end{notn}

  \begin{const}
 Consider the co-$t\Delta$-object in $2\cat$
 given by
 $$[m]\mapsto\mathcal O_2[m]\quad\text{ and }\quad[m]_t\mapsto\left\{
 \begin{array}{clc}
\mathbb E&m=1\\
\mathbb O_2[2]&m=2\\
 \mathcal O_2[m]&m\ge3
 \end{array}
 \right.$$
  with the usual cosimplicial structure, and with the further structure maps given by the canonical inclusions $[1]\to \mathbb{E}$ and $\cO_2[2] \to\mathbb{O}_2[2]$.
It induces an adjunction
 $$\honat_2\colon\psh{t\Delta}\rightleftarrows2\cat\colon\Nnat_2.$$
 We call $\Nnat_2$ the \emph{natural $2$-nerve} and $\honat_2$ the natural \emph{natural $2$-categorification}.
 \end{const}
  
  In particular, $\Nnat_2(\cC)$ is a $t\Delta$-set whose underlying simplicial set is the Street--Duskin nerve $\Nstreet_2(\cC)$, the marked $1$-simplices are those corresponding to $1$-equivalences in $\cC$ (each marked possibly multiple times), the marked $2$-simplices are those corresponding to $2$-isomorphisms of $\cC$ (each marked uniquely), and all $m$-simplices are marked for $m>2$ (each marked uniquely). More precisely, the $1$-simplex corresponding to a $1$-equivalence $f$ in $\cC$ is marked as many times as there are ways to complete $f$ to an adjoint equivalence $(f,g,\eta,\epsilon)$.

 What we have lost is fully faithfulness of the nerve construction on the nose, as shown by the following remark.
 
 \begin{rmk}
 Using the detailed description of $\honat_2(\Nnat_2(\cC))$ for any $2$-category $\cC$ given in \cref{prooflemma}, one can see that $\honat_2(\Nnat_2(\mathbb E))$ is a much larger $2$-category than $\mathbb E$; for instance, while the underlying $1$-category of $\mathbb E$ is generated by $f$ and $g$, with a careful analysis of $\honat_2(\Nnat_2(\mathbb E))$ one discovers that the underlying $1$-category of $\honat_2(\Nnat_2(\mathbb E))$ is freely generated by an infinite class of morphisms. In particular the counit
     \[
    \begin{tikzcd}
     \epsilon^{\natural}_{\mathbb E}\colon\honat_2(\Nnat_2(\mathbb E))\arrow[r, two heads, "\not\cong"]&\mathbb E
    \end{tikzcd}
    \]
    cannot be an isomorphism.
 \end{rmk}

 However, the following theorem guarantees that for $n=2$ the counit is an isomorphism at the level of homotopy categories, which is equivalent to saying that the nerve is fully faithful at the level of homotopy categories.

 \begin{thm}
 \label{counitweakequivalence}
For any $2$-category $\cC$ the (derived) counit
    \[
    \begin{tikzcd}
     \epsilon^{\natural}_{\cC}\colon\honat_2(\Nnat_2(\cC))\arrow[r, two heads, "\simeq"]&\cC
    \end{tikzcd}
    \]
     is an acyclic fibration in the Lack model structure, and in particular a biequivalence.
 \end{thm}
 
 The proof makes use of the following lemma, which will be proven in \cref{prooflemma}.
 
 \begin{lem}
 \label{localequivalence}
 For any $2$-category $\cC$ the (derived) counit is a local equivalence of categories, i.e., for any objects $x$ and $y$ of $\cC$ it induces an equivalence of categories
 \[
   \epsilon^{\natural}_{\cC}(x,y) \colon \Map_{\honat_2(\Nnat_2(\cC))}(x,y)\to \Map_{\cC}(x,y).
   \]
 \end{lem}
 
 We can now prove the theorem.
 
 \begin{proof}[Proof of \cref{counitweakequivalence}]
 We show that the counit $\epsilon^{\natural}_{\cC}$ has the right lifting property with respect to all generating cofibrations.
 \begin{itemize}[leftmargin=*]
     \item The fact that the counit lifts against $[-1]\hookrightarrow[0]$ is a consequence of the fact that the counit is bijective on objects.
     \item The fact that the counit lifts against $[0]\amalg[0]\hookrightarrow[1]$ is a consequence of the fact that the counit is surjective on $1$-morphisms and bijective on objects.
    \item The fact that the counit $\epsilon^\natural_{\cC}$ lifts against $\Sigma([0]\amalg[0])\hookrightarrow \Sigma{[1]}$ and $\Sigma[\rightrightarrows]\to\Sigma{[1]}$ is a consequence of the fact that
for every object $x$ and $y$ of $\cC$ the map induced by the counit on the hom-category $\epsilon^{\natural}_{\cC}(x,y)$ lifts against $[0]\amalg[0]\hookrightarrow[1]$ and $[\rightrightarrows]\to[1]$, which is true since $\epsilon^{\natural}_{\cC}(x,y)$ is fully faithful by \cref{localequivalence}.\qedhere
 \end{itemize}
 \end{proof}

  Motivated by \cref{modelstructureonSettransferred,modelstructuretransferred1}, we will show the following theorem in \cref{appendix1}.
 
  \begin{thm}
  \label{Lackistransferred}
    Lack's model structure on $2\cat$ is right-transferred from the Riehl--Verity model structure on $\psh{t\Delta}$ for $2$-complicial sets along the natural nerve $\Nnat_2\colon2\cat\to\psh{t\Delta}$.     In particular, the adjunction
   $$\honat_2\colon\psh{t\Delta}_{(\infty,2)}\leftrightarrows2\cat\colon\Nnat_2$$
   is a Quillen pair.
 \end{thm}

Combining \cref{counitweakequivalence,Lackistransferred}, we also obtain the following corollary.

\begin{cor}
\label{corcofibrant}
For any $2$-category $\cC$ the counit
$ \epsilon^{\natural}_{\cC}\colon\honat_2(\Nnat_2(\cC))\to\cC$
     is cofibrant replacement for $\cC$ in the Lack model structure.
\end{cor}

  While we already saw in \cref{counitweakequivalence} that the components of derived counit $\epsilon^\natural_{\cC}$ of the natural nerve-categorification adjunction at any $\cC$ is a weak equivalence, we do not expect the components of the derived unit to be weak equivalences, since we do not expect the natural nerve-categorification adjunction to be a Quillen equivalence. Indeed, here is a counterexample.
  
  \begin{rmk}
One can see that $\Nnat_2(\honat_2(\Delta[3]/\partial\Delta[3]))$ is isomorphic to $\Delta[0]$, and in particular the (derived) unit,
     \[
    \begin{tikzcd}
     \eta^{\natural}_{\Delta[3]/\partial\Delta[3]}\colon \Delta[3]/\partial\Delta[3]\arrow[r, two heads, "\not\cong"]&\Nnat_2(\honat_2(\Delta[3]/\partial\Delta[3]))\cong\Delta[0]
    \end{tikzcd}
    \]
     cannot be a weak equivalence since it is not a Kan weak equivalence between the underlying simplicial sets.
  \end{rmk}

 \addtocontents{toc}{\protect\setcounter{tocdepth}{1}}
\subsection{The natural nerve-categorification adjunction for $n>2$}

 Although it becomes more and more involved to formalize the natural nerve construction, it is easy to guess how the pattern would continue for higher values of $n$.
 
To this end, we first need to understand what is the correct notion of equivalence for an $m$-morphism in an $n$-category $\cC$. The idea is that an $m$-equivalence, as opposed to an $m$-isomorphism, should only be weakly invertible,
namely invertible up to higher coherence equivalences. This can be formalized coinductively (see e.g.~\cite[Def.6]{LMW} or \cite[\textsection1]{StreetOrientedSimplexes}), exploiting the higher categorical structure of $\cC$ as follows.
\begin{enumerate}[leftmargin=*]
    \item An $n$-morphism $\alpha\colon A\to B$ of $\cC$ is an \emph{$n$-equivalence} if it is an \emph{$n$-iso\-morphism}, namely if there exists an $n$-morphism $\beta\colon B\to A$ such that
     $$\id_A=\beta\alpha\colon A\to A\quad\text{ and }\quad\alpha\beta=\id_B\colon B\to B.$$
    \item For $0\le k\le n$, an $(n-k)$-morphism of $\cC$ is an \emph{$(n-k)$-equivalence} if there exists an $(n-k)$-morphism $\beta\colon B\to A$ together with $(n-k+1)$-equivalences

     $$\id_A\simeq\beta\alpha\colon A\to A\quad\text{ and }\quad\alpha\beta\simeq\id_B\colon B\to B.$$
\end{enumerate}

 \begin{notn}
For $n\ge0$, we denote by $\mathbb Adj_n$\footnote{The notation is justified by the fact that $\mathbb Adj_n$ is indeed the $n$-truncation of the $\omega$-category $\mathbb Adj$, that represents the \emph{homotopy coherent adjunction} for $\omega$-categories. This $\omega$-category is $\mathbf P^1$ following the terminology of \cite[\textsection4.7]{LMW}.} the \emph{free $n$-categorical $1$-equivalence}, i.e., the $n$-category with the property that $n$-functors $\mathbb Adj_n\to\cC$ into any $n$-category $\cC$ correspond to $1$-equivalences in $\cC$.
In particular, we recover as special instances the terminal category $\mathbb Adj_0\cong[0]$, the free $1$-isomorphism $\mathbb Adj_1\cong\mathbb I$ and the free $2$-categorical adjoint equivalence $\mathbb Adj_2\cong\mathbb E$.
More generally, for $n\ge0$ and $m\le n$ we denote by $\Sigma^{m-1}\mathbb Adj_{n-m+1}$ the \emph{free $n$-categorical $m$-equivalence}\footnote{The notation is justified by the fact that $\Sigma^{m-1}\mathbb Adj_{n-m+1}$ can in fact be understood as an iterated categorical suspension of the $(n-m+1)$-category $\mathbb Adj_{n-m+1}$.}. In particular, we recover the free $2$-isomorphism $\Sigma^{2-1}\mathbb Adj_{2-2+1}\cong\Sigma\mathbb I$.
 \end{notn}
 
We also need to fix  a notational convention for the free $m$-morphisms.

 \begin{notn}
For any $n\ge0$ and $m\le n$ we denote by $\Sigma^m[0]$ the \emph{free $m$-morphism of an $n$-category}, i.e., the $n$-category with the property that $n$-functors $\Sigma^m[0]\to\cC$ into any $n$-category $\cC$ correspond to $m$-morphisms in $\cC$. This $n$-category is in fact an $m$-category, and consists essentially of a $m$-morphism between two parallel $(m-1)$-morphisms\footnote{The $n$-category $\Sigma^{m}[0]$ appears often in the literature; Street \cite{StreetOrientedSimplexes} denotes it $2_m$, Joyal \cite{JoyalDisks} and Ara \cite{Ara} by $D_m$, and Lafont--M{\'e}tayer--Worytkiewicz \cite{LMW} by $\mathbf{O}^m$.}.
In particular, we recover the terminal category
 $\Sigma^0[0]\cong[0]$, the free $1$-morphism $\Sigma^1[0]\cong[1]$ and the free $2$-morphism $\Sigma^2[0]\cong\Sigma[1]$.
\end{notn}

With the notions of free $m$-equivalences and free $m$-morphisms of an $n$-category $\cC$ at our disposal, we can explain how to modify the Roberts-Street adjunction to obtain an adjunction with the desired homotopical properties.
In order to define the Roberts-Street adjunction, a crucial role was played by the $(m-1)$-truncated $m$-th oriental $\mathcal O_{m-1}[m]$. This $n$-category is by definition obtained by glueing an identity $m$-morphism on the top $m$-morphism of the $m$-th oriental $\mathcal O[m]$, and we can therefore express it as a pushout of $m$-categories
$$\mathcal O_{m-1}[m]\cong\mathcal O_n[m]\aamalg{\Sigma^m[0]}\Sigma^{m-1}[0].$$
Inspired by how we treated the cases $n=1,2$, rather than glueing an identity $m$-morphism $\Sigma^{m-1}[0]$ we instead glue a free $n$-categorical $m$-equivalence $\Sigma^{m-1}\mathbb Adj_{n-m+1}$. This idea is formalized as follows.

 \begin{notn}
 For $n\ge0$ and $m\le n$ we denote by
 $\mathbb O_n[m]$ the $n$-category obtained from $\mathcal{O}_n[m]$ by gluing an $m$-equivalence to the top morphism of $\mathcal O[m]$ regarded as an $n$-category. It can be expressed as the pushout
 $$\mathbb O_n[m]\cong\mathcal O_n[m]\aamalg{\Sigma^{m}[0]}\Sigma^{m-1}\mathbb Adj_{n-m+1}.$$
\end{notn}

  \begin{const}
 The co-$t\Delta$-object in $n\cat$
 given by
 $$[m]\mapsto\mathcal O_n[m]\quad\text{ and }\quad[m]_t\mapsto
 \left\{
 \begin{array}{clc}
\mathbb O_{n}[m]&m\le n\\
 \mathcal O_n[m]&m>n
 \end{array}
 \right.
 $$
 induces an adjunction
 $$\honat_n\colon\psh{t\Delta}\rightleftarrows n\cat\colon\Nnat_n.$$
  We call $\Nnat_n$ the \emph{natural $n$-nerve} and $\honat_n$ the natural \emph{natural $n$-categorification}.
 \end{const}
 
 In particular, $\Nnat_n(\cC)$ is a $t\Delta$-set whose underlying simplicial set is the Street nerve $\Nstreet_n(\cC)$ and the marked $m$-simplices are those witnessed by $m$-morphisms that are $m$-equivalences in $\cC$ (possibly multiple times), and $m$-simplices are marked in dimension higher than $m>n$ (each uniquely).
 
In this setup, we expect to be able to generalize the arguments presented in this and in the next section to the case of general $n$, and prove in forthcoming work the following conjectures.
  \begin{conjecture}
   The model structure on $n\cat$ is right-transferred from the Riehl--Verity model structure on $\psh{t\Delta}$ for $(\infty,n)$-categories along the natural nerve $\Nnat_n\colon n\cat\to\psh{t\Delta}$.
 \end{conjecture}
 
 \begin{conjecture}
 For any $n$-category $\cC$ the (derived) counit
   \[
    \begin{tikzcd}
     \epsilon^{\natural}_{\cC}\colon\honat_n(\Nnat_n(\cC))\arrow[r, two heads, "\simeq"]&\cC
    \end{tikzcd}
    \]
    is an acyclic fibration, and in particular an $n$-categorical equivalence.
 \end{conjecture}

 \section{Relations between the two nerve-categorification adjunctions}

We show that the natural nerve is a fibrant replacement of the Roberts-Street nerve.

\begin{const}
For any $2$-category $\cC$, there is a natural inclusion of $t\Delta$-sets
$$\Ntdelta_2(\cC)\hookrightarrow\Nnat_2(\cC)$$
given by taking the identity on the underlying simplicial sets and by regarding naturally the simplices witnessed by identities as marked simplices in $\Nnat_2(\cC)$.
\end{const}

We recall that an \emph{anodyne extension} is a map that can be expressed as a retract of a transfinite composition of pushouts of elementary anodyne extensions.

  \begin{thm} 
  \label{naturalnervefibrantreplacement}
  Let $\cC$ be a $2$-category.
  \begin{enumerate}[leftmargin=*]
  \item The nerve $\Ntdelta_2(\cC)$ is fibrant if and only if $\cC$ it has no non-identity $1$-isomorphisms and no non-identity $2$-isomorphisms.
\item The natural nerve $\Nnat_2(\cC)$ is fibrant.
      \item The canonical map $\Ntdelta_2(\cC)\to\Nnat_2(\cC)$ is an anodyne extension, and in particular a weak equivalence in the model structure for $(\infty,2)$-categories.
  \end{enumerate}
  \end{thm}
  
  Part of the combinatorics that goes into the proof was already established by Gurski in \cite{gurski}.
  
    \begin{proof}
    For part (1), we observe that if $\cC$ it has no weakly invertible invertible $1$-morphisms and no invertible $2$-morphisms the Roberts--Street nerve $\Ntdelta_2(\cC)$ agrees with the natural nerve $\Nnat_2(\cC)$, 
    which we will know to be fibrant after we will have shown Part (3). Viceversa, if $\cC$ contains a non-identity invertible $2$-morphism $\alpha\colon f\Rightarrow g$, the map
    $$\Delta[0]\star\eqDelta\to\Ntdelta_2(\cC)$$
from (\ref{stratifiedmap})
    does not lift against the saturation anodyne extension
    $$\Delta[0]\star\eqDelta\to\Delta[0]\star\Delta[3]^\sharp,$$
    so the Roberts--Street nerve $\Ntdelta_2(\cC)$ is not fibrant. Similarly, if $\cC$ contains a non-identity invertible $1$-morphism $f\colon x\to y$, the map
    $$\eqDelta\to\Ntdelta_2(\cC)$$
    whose image is given by
    \[
\simpfver{f}{f^{-1}}{f}{f}{\id_x}{\id_y}{\id_f}{\id_x}{\id_f}\simpfvercontinued{\id_y}{x}{y}{x}{y}.
\]
    does not lift against the saturation anodyne extension
    $$\eqDelta\to\Delta[3]^\sharp,$$
    so the Roberts--Street nerve $\Ntdelta_2(\cC)$ is not fibrant.
    
    Part (2) is a consequence of the fact that the natural nerve-categorification adjunction is a Quillen pair, as shown in \cref{naturaladjunctionQuillen}.
    
For part (3), we show how to build the natural nerve $\Nnat_2(\cC)$ from $\Ntdelta_2(\cC)$ by means of a (finite) composition
of retracts of pushouts of sums of elementary anodyne extensions.
Given that the inclusion $\Ntdelta_2(\cC)\to\Nnat_2(\cC)$ does not change the underlying simplicial set, to obtain $\Nnat_2(\cC)$ from $\Ntdelta_2(\cC)$ we only need add the (unique) marking of the $2$-simplices of $\Ntdelta_2(\cC)$ witnessed by $2$-isomorphisms of $\cC$ and the (generally not unique) marking of the $1$-simplices witnessed by $1$-equivalences of $\cC$.

Given any $2$-isomorphism $\alpha\colon f \Rightarrow g$ with inverse $\beta\colon g\Rightarrow f$, we can define a map
\begin{equation}
    \label{stratifiedmap}
\Delta[0]\star\eqDelta \to \Ntdelta_2(\cC)
\end{equation}
whose image is the $4$-simplex of the nerve $\Ntdelta_2(\cC)$ uniquely determined by the following $2$-skeleton
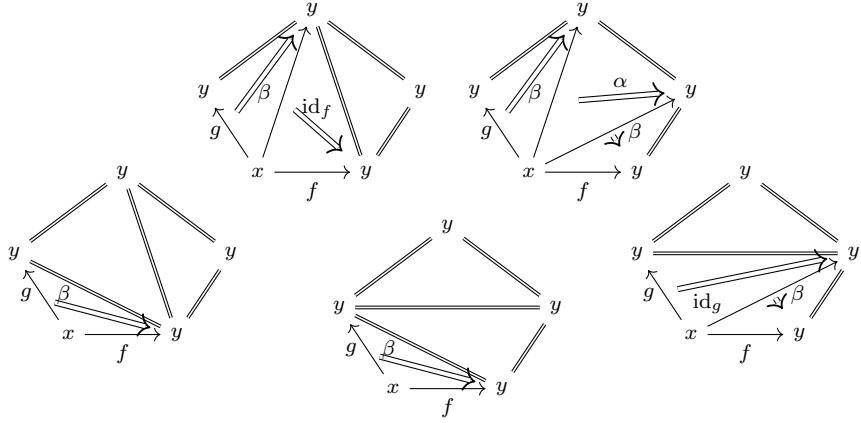
\begin{center}
 \begin{tikzpicture}[scale=0.8, font=\scriptsize]
  \def\l{1.8cm}
    \begin{scope}  
   \draw[fill] (0,0) node (b0){$x$}; 
  \draw[fill] (\l,0) node (b1){$y$};
   \draw[fill] (1.5*\l,0.75*\l)  node (b2){$y$};
   \draw[fill] (0.5*\l,1.5*\l) node (b3){$y$};
   \draw[fill] (-0.5*\l,0.75*\l) node (b4){$y$};

   \draw[<-] (b1)--node[below]{$f$}(b0);
   \draw[double] (b1)--(b2);
   \draw[double] (b2)--(b3);
   \draw[double] (b3)--(b4);
    \draw[<-] (b4)--node[left](A2){}node[left](A3){$g$}(b0);
   \draw[double] (b1)--(b3);
   \draw[<-] (b3)--node[left](A1){}(b0);

    \draw[twoarrowlonger] (A1)--node[above]{$\id_f$}(b1);
\draw[twoarrowlonger] (A2)--node[below]{$\beta$}(b3);
   \end{scope}
   
       \begin{scope}[xshift=2.5*\l] 
   \draw[fill] (0,0) node (b0){$x$}; 
  \draw[fill] (\l,0) node (b1){$y$};
   \draw[fill] (1.5*\l,0.75*\l)  node (b2){$y$};
   \draw[fill] (0.5*\l,1.5*\l) node (b3){$y$};
   \draw[fill] (-0.5*\l,0.75*\l) node (b4){$y$};

   \draw[<-] (b1)--node[below]{$f$}(b0);
   \draw[double] (b1)--(b2);
   \draw[double] (b2)--node[below](B2){}(b3);
   \draw[double] (b3)--(b4);
    \draw[<-] (b4)--node[left](B4){}node[left](B5){$g$}(b0);
   \draw[->] (b0)--node[left, xshift=0.2cm, yshift=-0.1cm](B1){}(b2);
   \draw[<-] (b3)--node[below](B3){}(b0);

    \draw[twoarrowlonger] (B3)--node[above]{$\alpha$}(b2);
     \draw[twoarrow] (B1)--node[above, xshift=0.15cm]{$\beta$}(b1);
     \draw[twoarrowlonger] (B4)--node[below]{$\beta$}(b3);
   \end{scope}
   
       \begin{scope}[xshift=4*\l, yshift=-1.5*\l]
   \draw[fill] (0,0) node (b0){$x$}; 
  \draw[fill] (\l,0) node (b1){$y$};
   \draw[fill] (1.5*\l,0.75*\l)  node (b2){$y$};
   \draw[fill] (0.5*\l,1.5*\l) node (b3){$y$};
   \draw[fill] (-0.5*\l,0.75*\l) node (b4){$y$};

   \draw[<-] (b1)--node[below]{$f$}(b0);
   \draw[double] (b1)--(b2);
   \draw[double] (b2)--(b3);
   \draw[double] (b3)--(b4);
    \draw[<-] (b4)--node[left](C2){}node[left](C3){$g$}(b0);
   \draw[->] (b0)--node[left, xshift=0.2cm, yshift=-0.1cm](C1){}(b2);
   \draw[double] (b2)--(b4);

\draw[twoarrow] (C1)--node[above, xshift=0.15cm]{$\beta$}(b1);
 \draw[twoarrowlonger] (C3)--node[below, near start]{$\id_g$}(b2);
   \end{scope}
   
       \begin{scope}[xshift=1.25*\l, yshift=-2*\l]
   \draw[fill] (0,0) node (b0){$x$}; 
  \draw[fill] (\l,0) node (b1){$y$};
   \draw[fill] (1.5*\l,0.75*\l)  node (b2){$y$};
   \draw[fill] (0.5*\l,1.5*\l) node (b3){$y$};
   \draw[fill] (-0.5*\l,0.75*\l) node (b4){$y$};

   \draw[<-] (b1)--node[below]{$f$}(b0);
   \draw[double] (b1)--(b2);
   \draw[double] (b2)--(b3);
   \draw[double] (b3)--(b4);
    \draw[<-] (b4)--node[left](D2){}node[left](D3){$g$}(b0);
   \draw[double] (b2)--(b4);
   \draw[double] (b4)--(b1);

\draw[twoarrowlonger] (D3)--node[above, xshift=-0.5cm]{$\beta$}(b1);
   \end{scope}
   
       \begin{scope}[xshift=-1.75*\l, yshift=-1.5*\l]
   \draw[fill] (0,0) node (b0){$x$}; 
  \draw[fill] (\l,0) node (b1){$y$};
   \draw[fill] (1.5*\l,0.75*\l)  node (b2){$y$};
   \draw[fill] (0.5*\l,1.5*\l) node (b3){$y$};
   \draw[fill] (-0.5*\l,0.75*\l) node (b4){$y$};

   \draw[<-] (b1)--node[below]{$f$}(b0);
   \draw[double] (b1)--(b2);
   \draw[double] (b2)--(b3);
   \draw[double] (b3)--(b4);

    \draw[<-] (b4)--node[left](E2){}node[left](E3){$g$}(b0);
   \draw[double] (b1)--(b3);
   \draw[double] (b1)--(b4);

\draw[twoarrowlonger] (E3)--node[above, xshift=-0.5cm]{$\beta$}(b1);
   \end{scope}
 \end{tikzpicture}
\end{center}
By taking the pushout of the sum of all such maps, as $\alpha$ varies,
of the saturation anodyne extensions
$$\Delta[0]\star\eqDelta \to\Delta[0]\star\Delta[3]^\sharp$$
we obtain an anodyne extension
$$\Ntdelta_2(\cC)\hookrightarrow P_1.$$
This inclusion does not change the underlying simplicial set, and $P_1$ is obtained from $\Nnat_2(\cC)$ by marking (multiple times) all $2$-simplices witnessed by an invertible $2$-morphism of $\cC$ and such with degenerate $0$th face.

Let $P_2$ be the stratified simplicial set obtained from $P_1$ by identifying multiple marking of the same simplex. By \cite[Lemma B.4]{or}, the map $$\Ntdelta_2(\cC)\hookrightarrow P_2$$
is a retract of the previous anodyne extension $\Ntdelta_2(\cC)\hookrightarrow P_1$, and is therefore also an anodyne extension.
In particular, $P_2$ is obtained from $\Nnat_2(\cC)$ by marking (uniquely) all $2$-simplices whose $0$-th face is degenerate witnessed by an invertible $2$-morphism of $\cC$.

Given any $2$-isomorphism $\alpha\colon f \Rightarrow g_2g_1$ with $g_2$ a non-identity morphism, we can define a map
$$\Delta^2[3]'\to P_2$$
whose image is
\[{\scriptsize
\simpfver{g_1}{g_2}{\id}{f}{g_2g_1}{g_2}{\alpha}{\id_{g_2g_1}}{\alpha}\simpfvercontinued{\id_{g_2}}{x}{y}{z}{z}}
\]
By taking the pushout of the sum of all such maps, as $\alpha$ varies, of the sum of the thinness anodyne extensions
$$\Delta^2[3]'\to\Delta^2[3]''$$
we obtain an anodyne extension
$$P_2\hookrightarrow P_3.$$
This inclusion does not change the underlying simplicial set, and $P_3$ is a stratified simplicial set obtained from $P_2$ by marking (uniquely) all remaining $2$-simplices witnessed by an invertible $2$-morphism of $\cC$.

Given any $1$-equivalence $f\colon x \to y$ of $\cC$, there exists $g\colon y\to x$ such that $f\colon x\rightleftarrows y\colon g$ is an adjoint equivalence in $\cC$, with invertible unit $\eta\colon\id_x\Rightarrow g f$ and invertible counit $\epsilon\colon fg\Rightarrow\id_y$. In particular the relation
 $$f*\eta=(\epsilon*f)^{-1}\colon f\Rightarrow fgf$$
holds. For any $(f,g,\eta,\epsilon)$
we can consider the map
$$
\eqDelta \to P_3
$$
given by
\[
\simpfver{f}{g}{f}{f}{\id_x}{\id_y}{\id_f}{\eta}{\id_f}\simpfvercontinued{\epsilon^{-1}}{x}{y}{x}{y}.
\]
By taking the pushout of the sum of all such maps, as $(f,g,\eta,\epsilon)$ varies, the sum of the saturation anodyne extension
$$\eqDelta\to\Delta[3]^\sharp$$
we obtain an anodyne extension
$$P_3\hookrightarrow P_4.$$
This inclusion does not change the underlying simplicial set, and $P_4$ is obtained from $P_3$ by marking (multiple times) all $1$-simplices witnessed by a $1$-equivalence of $\cC$.

Observe that $\Nnat_2(\cC)$ is a quotient of $P_4$, obtained by identifying all the markings of $f$ coming from $(f,g, \eta, \epsilon)$.  
Using a variant of \cite[Lemma B.4]{or}, the map $P_3\hookrightarrow \Nnat_2(\cC)$
is a retract of the previous anodyne extension $P_3\hookrightarrow P_4$, and is therefore also an anodyne extension, as desired.

One can now obtain the comparison of nerves as the composite of anodyne extensions
\[\Ntdelta_2(\cC)\hookrightarrow P_2\hookrightarrow P_3\hookrightarrow\Nnat_2(\cC).\qedhere\]
\end{proof}

 For any $t\Delta$-set $X$, one can write down a canonical map between the two $2$-categorifications,
  $$\honat_2(X)\to\hotdelta_2(X).$$
 This map is not in general a biequivalence, and here is an example.
 
 \begin{rmk}
 Let $X$ be a stratified set obtained by taking two copies of $\Delta[2]_t$ and identifying the boundaries $\partial\Delta[2]$. 
 Then one can check that the Roberts--Street $2$-categorification of $X$ is isomorphic to $[2]$, and the natural $2$-categorification of $X$ is isomorphic to two copies of $\mathbb O_2[2]$ of which we identify the boundaries (i.e., the underlying $1$-categories). In particular, the comparison map $\honat_2(X)\to\hotdelta_2(X)$ is not faithful on the hom-categories between $0$ and $2$.
 \end{rmk}

\section{Proof of \cref{Lackistransferred}}
 \label{appendix1}
The aim of this section if to prove \cref{Lackistransferred}. To this end, we will first show that the natural nerve-categorification adjunction is a Quillen pair, and use this fact to deduce that the model structure for $(\infty,2)$-categories can be transferred along the natural nerve $\Nnat_2$. Finally, we will show that the said transferred model structure coincides with Lack's model structure.
 
 \begin{prop}
 \label{naturaladjunctionQuillen}
 The adjunction $\honat_2\colon \psh{t\Delta}\rightleftarrows \twocat :\Nnat_2$
 is a Quillen pair.
 \end{prop}

 \begin{proof}
 We show in \cref{lemmacof,lemmaanodyne} that $\honat_2$ sends generating cofibrations to cofibrations and elementary anodyne extensions to acyclic cofibrations. We can then conclude by \cite[Prop. 2.4.40]{CisinskiBook} that $\honat_2$ is a left Quillen functor, as desired.
 \end{proof}

 \begin{lem}
 \label{lemmacof}
 The functor $\honat_2\colon \psh{t\Delta}\to\twocat$ sends generating cofibrations to cofibrations.
 \end{lem}
  
 \begin{proof}
 We show that $\honat_2$ sends all types of generating cofibrations to cofibrations.
   \begin{itemize}[leftmargin=*]
     \item The functor $\honat_2$ sends the canonical inclusion $\partial\Delta[0]\hookrightarrow\Delta[0]$ to the canonical inclusion $[-1]\hookrightarrow[0]$, which
     is a generating cofibration;
          \item The functor $\honat_2$ sends the canonical inclusion $\partial\Delta[1]\hookrightarrow\Delta[1]$ to the canonical inclusion $[0]\amalg[0]\hookrightarrow[1]$, which
     is a generating cofibration;
          \item The functor $\honat_2$ sends the canonical inclusion $\partial\Delta[2]\hookrightarrow\Delta[2]$ to a pushout of $[1]=\Sigma[0]\hookrightarrow\Sigma{[1]}$; this map is a composite of the canonical inclusion $[1]=\Sigma[0]\hookrightarrow \Sigma([0]\amalg[0])$, which is a pushout of the generating cofibration $[0]\amalg[0]\hookrightarrow[1]$ along itself, and of the canonical inclusion $\Sigma([0]\amalg[0])\hookrightarrow\Sigma{[1]}$, which is a generating cofibration.
            \item The functor $\honat_2$ sends the canonical inclusion $\partial\Delta[3]\hookrightarrow\Delta[3]$ to a pushout of the generating acyclic cofibration $\Sigma[\rightrightarrows]\longrightarrow\Sigma{[1]}$ along the map \linebreak $\Sigma[\rightrightarrows]\to\honat_2(\partial\Delta[3])$ whose image is given by the following two composite $2$-morphisms.
            \[
            \simpfbdrya{f_{01}}{}{f_{23}}{}{}{}{\alpha_{023}}{\alpha_{012}}{\alpha_{013}}\simpfbdryacontinued{\alpha_{123}}
            \]
                    \item The functor $\honat_2$ sends the canonical inclusion $\partial\Delta[m]\hookrightarrow\Delta[m]$ for $m\ge4$ to the identity of $\mathcal O_2[m]$.
                    To see this we first observe that, by comparing their right adjoints or by evaluating at representables, we find the relation of left adjoint functors
                    $$\honat_2\circ(-)^{\flat}\cong\hostreet_2\colon\sset\to 2\cat,$$
                    where
                    $(-)^\flat\colon\psh{\Delta}\to\psh{t\Delta}$ denotes the minimal stratification of a simplicial set (in which only degenerate simplices are marked).
                    As a consequence, $\honat_2$ sends the canonical inclusion $\partial\Delta[m]^\flat\hookrightarrow\Delta[m]^\flat$ for $m\ge4$ to $\hostreet_2(\partial\Delta[m])\to\hostreet_2(\Delta[m])$, namely the $2$-truncation of the canonical map of $\omega$-categories $\hostreet(\partial\Delta[m])\to\hostreet(\Delta[m])$. By \cite[Lemma 5.1]{StreetOrientedSimplexes}, this map is the inclusion of the underlying $(m-1)$-category of $\cO[m]$ (by \cite[Lemma 5.1]{StreetOrientedSimplexes}) into $\mathcal O[m]$, and in particular an isomorphism on the underlying $3$-categories.
                    Since the $2$-truncation construction only involves the underlying $3$-category of an $\omega$-category (cf. \cite[\textsection 6]{LMW}), the $2$-truncation of the considered map is therefore an isomorphism of $2$-categories, as desired.
          \item The functor $\honat_2$ sends the canonical inclusion $\Delta[1]\hookrightarrow\Delta[1]_t$ to the canonical inclusion $[1]\hookrightarrow \mathbb E$, which is easily seen to be a cofibration using Lack's characterization of cofibrations \cite[Prop.4.14]{lack1}.
          \item The functor $\honat_2$ sends the canonical inclusion $\Delta[2]\hookrightarrow\Delta[2]_t$ to the canonical inclusion $\mathcal O_2[2]\hookrightarrow\mathbb O_2[2]$, which is
          a pushout of the generating cofibration $\Sigma{[1]}\hookrightarrow\Sigma{\mathbb I}$.
      \item The functor $\honat_2$ sends by definition the canonical inclusion $\Delta[m]\hookrightarrow\Delta[m]_t$ for $m\ge 3$ to the identity of $\mathcal O_2[m]$.\qedhere
      \end{itemize}
\end{proof}

\begin{lem}
\label{lemmaanodyne}
 The functor $\honat_2\colon \psh{t\Delta}\to\twocat$ sends all elementary anodyne extensions to acyclic cofibrations.
\end{lem}

The following $2$-categorical fact is handy for the verifications regarding the outer complicial anodyne extensions.

 \begin{lem}
 \label{InverseModEqui}
          Let $\cC$ be any $2$-category, and $h\colon y\to z$ an equivalence. A $2$-morphism $\alpha \colon f\Rightarrow g$ such that $h*\alpha$ is invertible is itself invertible.
          \end{lem}

\begin{proof}[Proof of \cref{lemmaanodyne}]
      We now show that $\honat_2$ sends all types of elementary anodyne extensions to acyclic cofibrations.
   \begin{itemize}[leftmargin=*]
        \item The functor $\honat_2$ sends the complicial anodyne extensions $\Lambda^k[1]\to\Delta^k[1]$ for $k=0,1$ to $[0]\hookrightarrow\mathbb E$, which is a generating acyclic cofibration.
            \item The functor $\honat_2$ sends the complicial horn anodyne extension $\Lambda^1[2]\hookrightarrow\Delta^1[2]=\Delta[2]_t$ to the canonical inclusion $[2]\hookrightarrow\mathbb O_2[2]$, which is a pushout of the generating acyclic cofibration $[1]\to\Sigma{\mathbb I}$.
       \item The functor $\honat_2$ sends the complicial anodyne extensions $\Lambda^k[2]\to\Delta^k[2]$ for $k=0,2$ to a pushout of the generating acyclic cofibration $[1]\to\Sigma{\mathbb I}$.
         \item The functor $\honat_2$ sends the complicial anodyne extensions $\Lambda^k[3]\to\Delta^k[3]$ for $k=1,2$ to an isomorphism of a pushout of the form $\cO_2[3]\amalg_{\cO_2[2]}\mathbb O_2[2]$.
               \item The functor $\honat_2$ sends the complicial anodyne extensions $\Lambda^k[3]\to\Delta^k[3]$ for $k=0,3$ to an isomorphism of a pushout of the form $\cO_2[3]\amalg_{\cO_2[2]}\mathbb O_2[2]\amalg_{\cO_2[2]}\mathbb O_2[2]\amalg_{\Delta[1]} \mathbb{E}$.
                \item The functor $\honat_2$ sends the complicial anodyne extensions $\Lambda^k[4]\to\Delta^k[4]$ to an isomorphism, since the relation that defines $\honat_2(\Delta^k[m])$ starting from $\honat_2(\Lambda^k[m])$ in fact already holds.
      \item The functor $\honat_2$ sends the complicial anodyne extensions $\Lambda^k[m]\to\Delta^k[m]$ for $m\ge5$ and $0\le k\le m$ 
      to an isomorphism.
      Indeed, $\Delta^k[m]$ can be obtained from $\Lambda^k[m]$ by filling the boundary an $(m-1)$-simplex to an $(m-1)$-simplex, then by filling the boundary an $m$-simplex to an $m$-simplex and by finally marking the top newly added simplex. In particular, this means that the inclusion  $\Lambda^k[m]\hookrightarrow\Delta^k[m]$ is a composite of a pushout of $\partial\Delta[m-1]\hookrightarrow\Delta[m-1]$, followed by a pushout of $\partial\Delta[m]\hookrightarrow\Delta[m]$, followed by a pushout of $\Delta[m]\hookrightarrow\Delta[m]_t$, and by previous considerations these types of maps are all sent to identities by $\honat_2$.
       \item The functor $\honat_2$ sends the thinness anodyne extensions $\Delta^k[2]'\to\Delta^k[2]''$  $k=0,1,2$ to an isomorphism, which can be seen using the fact adjoint equivalences satisfy the two-out-of-three property.
                        \item The functor $\honat_2$ sends the thinness anodyne extensions $\Delta^k[3]'\to\Delta^k[3]''$ for $1\le k\le m$ to an isomorphism, which can be seen using the fact that $2$-isomorphisms satisfy the two-out-of-three property.
      \item The functor $\honat_2$ sends the thinness anodyne extensions $\Delta^k[m]'\to\Delta^k[m]''$ to an identity for $m\geq 4$. Indeed, $\Delta^k[m]''$ can be obtained from $\Delta^k[m]'$ by the marking of the $k$-th face. This means that the extension $\Delta^k[m]'\to\Delta^k[m]''$ is a pushout of $\Delta[m-1]\to \Delta[m-1]_t$,
      which is mapped by $\honat_2$ to an identity for $m\geq 4$. 
      \item As already pointed out, the functor $\honat_2$ sends the $m$-th triviality anodyne extension $\Delta[m]\hookrightarrow\Delta[m]_t$ for $m\ge 3$ to the identity of $\mathcal O_2[m]$.
      \item The functor $\honat_2$ sends the minus first saturation anodyne extension 
$\Delta[3]_{eq}  \to\Delta[3]^{\sharp}$ to an isomorphism, which can be seen using the fact that adjoint equivalences satisfy the two-out-of-six property.
      \item The functor $\honat_2$ sends the $0$-th saturation anodyne extension 
$\Delta[0]\star\Delta[3]_{eq}  \to\Delta[0]\star\Delta[3]^{\sharp}$ to an to an isomorphism, which can be seen using the fact that $2$-isomorphisms satisfy the two-out-of-six property.
            \item The functor $\honat_2$ sends the $l$-th saturation anodyne extension 
      $ \Delta[l]\star \Delta[3]_{eq}  \to \Delta[l]\star \Delta[3]^{\sharp}$
     for $l\ge2$ to an isomorphism, which can be seen inductively.\qedhere
      \end{itemize}
\end{proof}
 
  \begin{prop}
  \label{existencetransferred}
    The category $2\cat$ supports the right-transferred model structure of the Riehl--Verity model structure on $\psh{t\Delta}$ for $2$-complicial sets along the natural nerve $\Nnat_2\colon2\cat\to\psh{t\Delta}$. In this model structure, the fibrations and the weak equivalences are created by $\Nnat_2$ and the generating cofibrations are obtained by taking the image via $\honat_2$ of the generating cofibrations of the model structure for $(\infty,2)$-categories on $\psh{t\Delta}$.
 \end{prop}

 \begin{proof}
 The model structure for $2$-complicial sets is cofibrantly generated and the category $2\cat$ is complete and cocomplete. We now check that the conditions of the transfer theorem \cite[Theorem 11.3.2]{Hirschhorn} hold for the desired adjunction.
\begin{enumerate}[leftmargin=*]
\item[(1)] Since $\twocat$ is a locally presentable (e.g.\ in \cite[\textsection 6]{LMW}), we can use \cite[Prop.5.2.10]{BorceuxHandbook2} to conclude that every object is small with respect to some cardinal.
\item[(2)]
Although we do not know have an explicit description of the generating acyclic cofibrations of the model structure for $(\infty,2)$-categories,
since $\honat_2$ is left Quillen any (retract of) transfinite composition of push\-outs of images of generating acyclic cofibrations under $\honat_2$ is in 
particular an acyclic cofibration, and since all $2$-categories are fibrant it is sent by $\Nnat_2$ to a weak equivalence.\qedhere
\end{enumerate}
 \end{proof}

   \begin{prop}
 The right-transferred model structure on $2\cat$ coincides with Lack's model structure.
 \end{prop}

\begin{proof}
We observe that in both model structures all objects are fibrant; and the generating cofibrations of the two model structures generate the same set of cofibrations. Indeed, we already know from \cref{existencetransferred}
that the (generating) cofibrations of the transferred model structure are Lack cofibrations and we now show that Lack cofibrations are transferred cofibrations.
\begin{itemize}[leftmargin=*]
    \item The Lack generating cofibration $[-1]\to[0]$ is a generating transferred cofibration, obtained as
    categorification of the canonical inclusion $\partial\Delta[0]\hookrightarrow\Delta[0]$.
    \item The Lack generating cofibration $[0]\amalg[0]\to[1]$ is a generating transferred cofibration, obtained as the categorification of the canonical inclusion $\partial\Delta[1]\hookrightarrow\Delta[1]$.
    \item The Lack generating cofibration $\Sigma([0]\amalg[0])\to\Sigma[1]$ is a retract of the generating transferred cofibration obtained as the categorification of $\partial\Delta[2]\hookrightarrow\Delta[2]$.
 We now show how to express $\Sigma([0]\amalg[0])$ as a retract of the categorification of $\partial\Delta[2]$. First, we take the inclusion
    $\Sigma([0]\amalg[0])\to \honat_2(\partial\Delta[2])$,
    that sends the two non-identity morphisms $f$ and $g$ of $\Sigma([0]\amalg[0])$ to the two $1$-morphisms $f_{02}$ and $f_{12}f_{01}$ in $\honat_2(\partial\Delta[2])$
    Then, we take the map
    determined by sending the $1$-morphisms $f_{01}$, $f_{02}$ and $f_{12}$ of $\honat_2(\partial\Delta[2])$ to the $1$-morphisms $\id_{x}$, $f$ and $g$ of $\Sigma([0]\amalg[0])$, respectively.
    To express $\Sigma[1]$ as a retract of the categorification of $\Delta[2]$, we observe that the map $\Sigma([0]\amalg[0])\to \honat_2(\partial\Delta[2])$ extends uniquely to a map $\Sigma[1]\to\honat_2(\Delta[2])$, and that the map $\honat_2(\partial\Delta[2])\to\Sigma([0]\amalg[0])$ extends uniquely to a map $\honat_2(\Delta[2])\to\Sigma[1]$.
    It is easy to check that the two new maps compose to an identity of $\Sigma[1]$, as desired.
     \item The Lack generating cofibration $\Sigma[\rightrightarrows]\to\Sigma[1]$ is a retract of the generating transferred cofibration obtained as the categorification of $\partial\Delta[3]\hookrightarrow\Delta[3]$, as follows. We now show how to express $\Sigma[\rightrightarrows]$ as a retract of the categorification of $\partial\Delta[3]$. We first take the inclusion
    $\Sigma[\rightrightarrows]\to \honat_2(\partial\Delta[3])$,
    that sends the two non-identity $2$-morphisms $\alpha$ and $\beta$ of $\Sigma[\rightrightarrows]$ to the $2$-morphisms $(f_{23}*\alpha_{012})\alpha_{023}$ and $(\alpha_{123}*f_{01})\alpha_{013}$ of $\honat_2(\partial\Delta[3])$, where $f_{ij}$ is the edge of $\Delta[3]$ with vertices $i$ and $j$ and $\alpha_{ijk}$ is the $2$-simplex of $\Delta[3]$ with vertices $i$, $j$ and $k$.
 Then, we take the map
    $\honat_2(\partial\Delta[3])\to\Sigma[\rightrightarrows]$
   determined by sending 
   the four $2$-morphisms represented by the $2$-simplices $\alpha_{012}$, $\alpha_{013}$, $\alpha_{023}$ and $\alpha_{123}$ of $\partial\Delta[3]$ to the $2$-morphisms $\id_{\id_x}$, $\beta$, $\alpha$, and $\id_g$, respectively.
    To express $\Sigma[1]$ as a retract of the categorification of $\Delta[3]$, we observe that the map $\Sigma[\rightrightarrows]\to\honat_2(\partial\Delta[3])$ induces a map $\Sigma[1]\to\honat_2(\Delta[3])$, and that the map $\honat_2(\partial\Delta[3])\to\Sigma[\rightrightarrows]$ extends uniquely to a map $\honat_2(\Delta[3])\to\Sigma[1]$.
    It is easy to check that the two new maps compose to an identity of $\Sigma[\rightrightarrows]$, as desired. 
\end{itemize}
We conclude recalling from \cite[Prop.E.1.10]{JoyalVolumeII} that a model structure is determined by cofibrations and fibrant objects.
\end{proof}

\section{Proof of \cref{localequivalence}}

\label{prooflemma}

We now prove that the counit of the natural nerve-categorification adjunction is a homwise equivalence of categories. 
The proof requires a careful analysis of $\honat_2(\Nnat_2(\cC))$ and of the effect of the counit at the level of hom-categories.

Roughly speaking, the way one builds the $2$-category $\honat_2(\Nnat_2(\cC))$ is by using the data already present in $\cC$, and by adding formal inverses to $2$-isomorphisms, and by adding formal inverse equivalences to $1$-equivalences of $\cC$. Then, formal inverses of $2$-isomorphisms get identified with the inverses already present in $\cC$, and all formal inverse equivalences are turn out to be all isomorphic to each other in a canonical way.
We now investigate more precisely the $2$-categorical structure of $\honat_2(\Nnat_2(\cC))$.

\begin{enumerate}[leftmargin=*]
    \item[(0)] Any object of $\honat_2(\Nnat_2(\cC))$ is uniquely described as an object $x$ of $\cC$.
    \item The underlying $1$-category of  $\honat_2(\Nnat_2(\cC))$ is generated by the following types of morphisms:
    \begin{itemize}[leftmargin=*]
        \item a morphism $[f]\colon x\to y$, for any such morphism $f$ in $\cC$; and
        \item for any adjoint equivalence $(f,g,\eta,\epsilon)$ in $\cC$ with $f\colon x\to y$, a morphism $\widetilde{f}_{(f,g,\eta,\epsilon)}\colon x\to y$ and a morphism $\widetilde{g}_{(f,g,\eta,\epsilon)}\colon y\to x$.
    \end{itemize} 
    These generators are subject to the relations that $[f]$ equals $\widetilde{f}_{(f,g,\eta,\epsilon)}$ and that both $\widetilde{g}_{(\id_x,\id_x,\id_{\id_x},\id_{\id_x})}$ and $[\id_x]$ equal $\id_x$.
    In particular, any $1$-morphism of $\honat_2(\Nnat_2(\cC))$ is uniquely described (modulo identities) by a word $a_1a_2\dots a_n$, with $a_i=[f]$ being represented by a morphism of $\cC$ or $a_i=\widetilde{g}_{(f,g,\eta,\epsilon)}$.
    \item The $2$-categorical structure of $\honat_2(\Nnat_2(\cC))$ is generated by the following types of $2$-morphisms:
    \begin{itemize}[leftmargin=*]
        \item for any $2$-morphism $\varphi\colon c \Rightarrow d$ in $\cC$ and any decomposition $d=d_1d_2$ of $1$-morphisms in $\cC$, the $2$-morphism
        $$\varphi_{d_1,c, d_2}\colon [c]\Rightarrow [d_1][d_2]$$ and, if $\varphi$ is $2$-isomorphism of $\cC$, its inverse
        $$\varphi^{-1}_{d_1,c,d_2}\colon [d_1][d_2]\Rightarrow [c];$$
        \item for any $1$-equivalence $f\colon x\to y$ in $\cC$, the unit and counit $2$-morphisms
        $$\widetilde{\eta}_{(f,g,\eta,\epsilon)}\colon \id_x \Rightarrow \widetilde{g}_{(f,g,\eta,\epsilon)}[f]\quad\text{ and }\quad\widetilde{\epsilon}_{(f,g,\eta,\epsilon)} \colon [f]\widetilde{g}_{(f,g,\eta,\epsilon)} \Rightarrow \id_y$$
        and their inverses
        $$\widetilde{\eta}_{(f,g,\eta,\epsilon)}^{-1}\colon \widetilde{g}_{(f,g,\eta,\epsilon)}[f]\Rightarrow \id_x \quad\text{ and }\quad\widetilde{\epsilon}_{(f,g,\eta,\epsilon)}^{-1} \colon  \id_y\Rightarrow [f]\widetilde{g}_{(f,g,\eta,\epsilon)}.$$
    \end{itemize}
    In particular, any $2$-morphism of $\honat_2(\Nnat_2(\cC))$ is described by a word composed by these generators, subject to the following relations: the relations encoded in the $3$-simplices of $\Nnat_2(\cC)$, the triangle identities for $\widetilde{\eta}_{(f,g,\eta,\epsilon)}$ and $\widetilde{\epsilon}_{(f,g,\eta,\epsilon)}$, the relations witnessing inverses for $\varphi_{d_1,c,d_2}$, $\widetilde{\eta}_{(f,g,\eta,\epsilon)}$ and $\widetilde{\epsilon}_{(f,g,\eta,\epsilon)}$, the identification of $(\id_d)_{\id_x,d, d}$ and $(\id_d)_{d, d,\id_x}$ of with $\id_{[d]}$, the identification of $\widetilde{\eta}_{(\id_x,\id_x,\id_{\id_x},\id_{\id_x})}$ and $\widetilde{\epsilon}_{(\id_x,\id_x,\id_{\id_x},\id_{\id_x})}$ with $\id_{[d]}$.
\end{enumerate}
According to this description, the counit
$$\epsilon^{\natural}_{\cC}\colon\honat_2(\Nnat_2(\cC))\to\cC$$
sends
\begin{itemize}[leftmargin=*]
    \item the objects and the generators of type $[f]$ and $\varphi_{d_1,c,d_2}$ to the corresponding to objects or $1$- or $2$-morphisms of $\cC$, namely $f$ or $\varphi$.
    \item the generators of type $\widetilde{g}_{(f,g,\eta,\epsilon)}$, $\widetilde{\eta}_{(f,g,\eta,\epsilon)}$ and $\widetilde{\epsilon}_{(f,g,\eta,\epsilon)}$ that complete $1$-equivalences to an adjoint equivalence to $g$, $\eta$ and $\epsilon$, respectively.
     \item the generators of type, $\widetilde{\eta}_{(f,g,\eta,\epsilon)}^{-1}$, $\widetilde{\epsilon}_{(f,g,\eta,\epsilon)}^{-1}$, $\varphi_{c,d_1,d_2}^{-1}$ to the inverses in $\cC$ of $\widetilde{\eta}_{(f,g,\eta,\epsilon)}$ and $\widetilde{\epsilon}_{(f,g,\eta,\epsilon)}$ and $\varphi_{c,d_1,d_2}$, respectively.
\end{itemize}

We can now prove the lemma.

 \begin{proof}[Proof of \cref{localequivalence}]
 
 We construct an inverse equivalence to the map induced by the counit on hom-categories. Consider the functor
 $$F_{x,y}\colon\Map_{\cC}(x,y)\to\Map_{\honat_2(\Nnat_2(\cC))}(x,y)$$
 that sends any $1$-morphism $f\colon x\to y$ of $\cC$ to the $1$-morphism of $\honat_2(\Nnat_2(\cC))$ represented by $f$, and any $2$-morphism $\varphi\colon a\Rightarrow b$ of $\cC$ to the $2$-morphism of $\honat_2(\Nnat_2(\cC))$ represented by $\varphi_{b,a,\id_x}$
 A careful analysis of the relations between $2$-morphisms of $\honat_2(\Nnat_2(\cC))$ shows that it is a functor, and by definition
$$\epsilon^{\natural}_{\cC}(x,y)\circ F_{x,y}=\id_{\Map_{\cC}(x,y)}$$
 We now construct a natural isomorphism
 $$\Psi\colon \id_{\Map_{\honat_2(\Nnat_2(\cC))}(x,y)}\Rightarrow F_{x,y}\circ \epsilon^{\natural}_{\cC}(x,y),$$
 by first constructing each component of $\Psi$ and then proving that it is natural.
 
For $r\ge 1$ we construct inductively a family of $2$-isomorphisms in $\honat_2(\Nnat_2(\cC))$ of the form
   \[
   \Psi_{a_1\ldots a_r}\colon a_1\ldots a_r \Rightarrow F_{x,y}(\epsilon^{\natural}_{\cC}(x,y)(a_1\ldots a_r)),
   \]
   where $a_i=[f]$ or $a_i=\widetilde{g}_{(f,g,\eta,\epsilon)}$.
   \begin{itemize}[leftmargin=*]
       \item For $r=1$, if $a_1=[f]$ for a $1$-morphism $f\colon x\to y$ of $\cC$ we set
       $$\Psi_{[f]}:=\id_{f}\colon [f]\Rightarrow [f],$$
       and if $a_1=\widetilde{g}_{(f,g,\eta,\epsilon)}$ is a formal adjoint for a $1$-equivalence $f\colon x\to y$ of $\cC$,
       we set $\Psi_{\widetilde{g}_{(f,g,\eta,\epsilon)}}$ to be the composite
           \[
    \begin{tikzcd}[column sep=3cm]
    \Psi_{\widetilde{g}_{(f,g,\eta,\epsilon)}}\colon {\widetilde{g}_{(f,g,\eta,\epsilon)}} \arrow[r, Rightarrow, "{\eta_{g,\id_x,f}*\widetilde{g}_{(f,g,\eta,\epsilon)}}"] & {[g][f]\widetilde{g}_{(f,g,\eta,\epsilon)}} \arrow[r, Rightarrow, "{[g]*\widetilde{\epsilon}_{(f,g,\eta,\epsilon)}}"] &{[g]}.
    \end{tikzcd}
    \]
       \item    For $r>1$, assume $\Psi_{a_1\ldots a_{r-1}}$ to be already defined. Recall that for any pair of composable $1$-morphisms $d_1$ and  $d_2$ in $\cC$
       there is a $2$-isomorphism in $\honat_2(\Nnat_2(\cC))$ corresponding to $\id_{d_1d_2}\colon d_1d_2\Rightarrow d_1d_2$, which we denote
       $$ I_{d_1,d_2}:=(\id_{d_1d_2})_{d_1, (d_1d_2), d_2} \colon [d_1d_2] \Rightarrow [d_1][d_2].$$
       We also write as a shorthand $$b:=\epsilon^{\natural}_{\cC}(x,y)(a_1\ldots a_{r-1})\text{ and }c:=\epsilon^{\natural}_{\cC}(x,y)(a_r).$$
       Set $\Psi_{a_1\ldots a_r}$ to be the composite
     \[
   \begin{tikzcd}
    a_1\ldots a_{r-1}a_r \arrow[r, Rightarrow, "{\Psi_{a_1\ldots a_{r-1}}*a_r}"] &[+0.8cm] {[b]a_r} \arrow[r, Rightarrow, "{[b]*\Psi_{a_r}}"]& {[b][c]} \arrow[r, Rightarrow, "{ I_{b,c}^{-1}}" ]&{F_{x,y}(\epsilon^{\natural}_{\cC}(x,y)(a_1\ldots a_r))}. 
   \end{tikzcd}
   \]
   \end{itemize}
  We observe that the associativity of composition for $1$-morphisms in $\cC$ yields $3$-simplices in $\Nnat_2(\cC)$ which in turn impose the following relations on $ I$ in $\honat_2(\Nnat_2(\cC))$ inductively for $i=1,\dots,r$:
   \begin{equation}
\label{psiandbeta}
\Psi_{a_1\ldots a_r}= I^{-1}_{\epsilon^{\natural}_{\cC}(x,y)(a_1\ldots a_i),\epsilon^{\natural}_{\cC}(x,y)(a_{i+1}\ldots a_r)}\circ (\Psi_{a_1\ldots a_i}*\Psi_{a_{i+1}\ldots a_r}).
   \end{equation}

   In order to show that $\Psi$ is natural, thanks to the relation (\ref{psiandbeta}) it suffices to check that is natural with respect to the generating $2$-morphisms of $\honat_2(\Nnat_2(\cC))$ of the form $\varphi_{d_1,c,d_2}$, $\widetilde{\eta}_{(f,g,\eta,\epsilon)}$, $\widetilde{\epsilon}_{(f,g,\eta,\epsilon)}$, and their inverses.
   \begin{itemize}[leftmargin=*]
       \item Given any $2$-morphism $\varphi \colon c\Rightarrow d_1d_2$ of $\cC$, the naturality square of $\Psi$ on $\varphi_{c,d_1d_2,\id_x}$,
   \[
   \begin{tikzcd}
    {[c]} \arrow[d, Rightarrow, "{\varphi_{d_1,c,d_2}}"swap] \arrow[r, Rightarrow, "{\Psi_{c}}"] & {[c]} \arrow[d, Rightarrow, "{\varphi_{d_1d_2,c,\id_x}}"]\\
    {[d_1][d_2]} \arrow[r, Rightarrow, "{\Psi_{d_1d_2}}"swap] & {[d_1d_2]},
   \end{tikzcd}
   \]
   commutes thanks to the relation on $2$-morphisms of $\honat_2(\Nnat_2(\cC))$ witnessed by following the $3$-simplex of $\Nnat_2(\cC)$
      \[
   \simpfver{\id_x}{d_2}{d_1}{c}{d_2}{d_1d_2}{\varphi}{\id_{d_2}}{\varphi}\simpfvercontinued{\id_{d_1d_2}}{x}{x}{y}{z}.
   \]
   \item Given any $1$-equivalence $f\colon x\to y$ of $\cC$, the naturality square of $\Psi$ on $\widetilde{\eta}_{(f,g,\eta,\epsilon)}$ is the diagram
         \[
   \begin{tikzcd}
    \id_x \arrow[rd, Rightarrow, "{\eta_{g, \id_x, f}}"] \arrow[ddd, Rightarrow, "{\widetilde{\eta}_{(f,g,\eta,\epsilon)}}"swap] \arrow[rrr, Rightarrow, "{\Psi_{\id_x}}"] &&[+0.5cm]& {\id_x} \arrow[ddd, Rightarrow, "{\eta_{gf,\id_x,\id_x}}"]\\
    &{[g][f]}\arrow[d, Rightarrow, "{[g][f]*\widetilde{\eta}_{(f, g, \eta, \epsilon)}}"] \arrow[rd, equals, bend left]&&\\
     & {[g][f]\widetilde{g}_{(f,g,\eta,\epsilon)}[f]} \arrow[r, Rightarrow, "{[g]\widetilde{\epsilon}_{(f, g, \eta, \epsilon)}[f]}"swap]& {[g][f]}\arrow[rd, Rightarrow, "I_{g,f}^{-1}"]&\\
    {\widetilde{g}_{(f,g,\eta,\epsilon)}[f]}\arrow[ru, Rightarrow, "{\eta_{g,\id_x, f}*{\widetilde{g}_{(f,g,\eta,\epsilon)}[f]}}"swap]  \arrow[rrr, Rightarrow, "{\Psi_{\widetilde{g}_{(f,g,\eta,\epsilon)}[f]}}"swap] &&& {[gf]},
   \end{tikzcd}
   \]
   and commutes thanks to the triangle identities for $\widetilde{\eta}_{(f,g,\eta,\epsilon)}$ and $\widetilde{\epsilon}_{(f,g,\eta,\epsilon)}$, and thanks to the definition of $\Psi_{\widetilde{g}_{(f,g,\eta,\epsilon)}[f]}$.

   \item  Given any $1$-equivalence $f\colon x\to y$ of $\cC$, the naturality square of $\Psi$ on $\widetilde{\epsilon}_{(f,g,\eta,\epsilon)}$ is the diagram
        \[
  \begin{tikzcd}
  {[f]\widetilde{g}_{(f,g,\eta,\epsilon)}} 
  \arrow[rd, Rightarrow, "{[f]*\eta_{ g, \id_x, f}*\widetilde{g}_{(f,g,\eta,\epsilon)}}"]
  \arrow[rdd, Rightarrow, "{\scriptsize [f]*\eta_{ gf, \id_x, \id_x}*\widetilde{g}_{(f,g,\eta,\epsilon)}}"{sloped, near start}]
  \arrow[rddd, Rightarrow, "{(f*\eta)_{fgf, f, \id_x}*\widetilde{g}_{(f,g,\eta,\epsilon)}}"{swap, near end, sloped}]
  \arrow[rdddd, equals, bend right]
  \arrow[ddddd, Rightarrow, "{\widetilde{\epsilon}_{(f,g,\eta,\epsilon)}}"{near end}]
  \arrow[rrr, Rightarrow, "{\Psi_{[f]\widetilde{g}_{(f,g,\eta,\epsilon)}}}"] 
  &[+1.2cm]&[+0.4cm]& {[fg]} \arrow[ddddd, Rightarrow, "{\epsilon_{\id_y,fg,\id_y}}"swap]\\
  &{[f][g][f]\widetilde{g}_{(f,g,\eta,\epsilon)}}
  \arrow[r, Rightarrow, "{[f][g]*\widetilde{\epsilon}_{(f, g, \eta, \epsilon)}}"]
  \arrow[d, Rightarrow, "{[f]*I_{g,f}^{-1}*\widetilde{g}_{(f,g,\eta,\epsilon)}}"]
  &{[f][g]}\arrow[ru, Rightarrow, "{I_{f,g}^{-1}}"]&\\
  &{[f][gf]\widetilde{g}_{(f,g,\eta,\epsilon)}}\arrow[d, Rightarrow, "{I_{f,gf}*\widetilde{g}_{(f,g,\eta,\epsilon)}}"]&&\\
  &{[fgf]\widetilde{g}_{(f,g,\eta,\epsilon)}}
  \arrow[d, Rightarrow, "{(\epsilon*f)_{f, fgf,\id_x}*\widetilde{g}_{(f,g,\eta,\epsilon)}}"]
  &{[fg][f]\widetilde{g}_{(f,g,\eta,\epsilon)}} \arrow[ld, Rightarrow, bend left, "{\epsilon_{\id_y, fg, \id_y}*[f]\widetilde{g}_{(f,g,\eta,\epsilon)}}"]\arrow[l, Rightarrow, "{I^{-1}_{fg,f}*\widetilde{g}_{(f,g,\eta,\epsilon)}}"swap]\arrow[uuur, Rightarrow, "{[fg]*\widetilde{\epsilon}_{(f,g,\eta,\epsilon)}}"{near start}]&\\
    &
    {[f]\widetilde{g}_{(f,g,\eta,\epsilon)}}\arrow[ld, Rightarrow, "{\widetilde{\epsilon}_{(f, g, \eta, \epsilon)}}"]
    &&\\
    \id_y  \arrow[rrr, Rightarrow, "{\Psi_{\id_y}}"] &&& {[\id_y]} 
  \end{tikzcd}
  \]
   and commutes thanks to the triangle identities for $\eta$ and $\epsilon$ in $\cC$ 
   and the relation (\ref{psiandbeta}).
   \item Finally, the commutativity naturality square of $\Psi$ on the inverses of $\widetilde{\eta}_{(f,g,\eta,\epsilon)}$, $\widetilde{\epsilon}_{(f,g,\eta,\epsilon)}$ and $\varphi_{d_1, c, d_2}$ follows from that of the commutativity squares on $\widetilde{\eta}_{(f,g,\eta,\epsilon)}$, $\widetilde{\epsilon}_{(f,g,\eta,\epsilon)}$ and $\varphi_{d_1,c,d_2}$, respectively.
  \qedhere
   \end{itemize}
 \end{proof}

\bibliographystyle{amsalpha}
\bibliography{ref}

\providecommand{\bysame}{\leavevmode\hbox to3em{\hrulefill}\thinspace}
\providecommand{\MR}{\relax\ifhmode\unskip\space\fi MR }
\providecommand{\MRhref}[2]{%
  \href{http://www.ams.org/mathscinet-getitem?mr=#1}{#2}
}
\providecommand{\href}[2]{#2}
\begin{thebibliography}{LMW10}

\bibitem[AC19]{CamarenaSets}
Omar Antol\'{i}n-Camarena, \emph{The nine model category structures on the
  category of sets}, available at
  \url{https://www.matem.unam.mx/~omar/notes/modelcatsets.html}, retrieved in
  Feb 2019.

\bibitem[AM14]{AraMaltsiniotisVers}
Dimitri Ara and Georges Maltsiniotis, \emph{Vers une structure de cat\'{e}gorie
  de mod\`eles \`a la {T}homason sur la cat\'{e}gorie des {$n$}-cat\'{e}gories
  strictes}, Adv. Math. \textbf{259} (2014), 557--654. \MR{3197667}

\bibitem[Ara14]{Ara}
Dimitri Ara, \emph{Higher quasi-categories vs higher {R}ezk spaces}, J.
  K-Theory \textbf{14} (2014), no.~3, 701--749. \MR{3350089}

\bibitem[Bor94]{BorceuxHandbook2}
Francis Borceux, \emph{Handbook of categorical algebra. 2}, Encyclopedia of
  Mathematics and its Applications, vol.~51, Cambridge University Press,
  Cambridge, 1994, Categories and structures. \MR{1313497}

\bibitem[Cam19]{campbell}
Alexander Campbell, \emph{A homotopy coherent cellular nerve for bicategories},
  draft available at
  \url{http://web.science.mq.edu.au/~alexc/coherentnerve.pdf}, 2019.

\bibitem[Cis19]{CisinskiBook}
Denis-Charles Cisinski, \emph{Higher categories and homotopical algebra},
  Cambridge Studies in Advanced Mathematics, Cambridge University Press, 2019.

\bibitem[CP86]{CordierPorterHomotopyCoherent}
Jean-Marc Cordier and Timothy Porter, \emph{Vogt's theorem on categories of
  homotopy coherent diagrams}, Math. Proc. Cambridge Philos. Soc. \textbf{100}
  (1986), no.~1, 65--90. \MR{838654}

\bibitem[Dus02]{duskin}
John~W. Duskin, \emph{Simplicial matrices and the nerves of weak
  {$n$}-categories. {I}. {N}erves of bicategories}, Theory Appl. Categ.
  \textbf{9} (2001/02), 198--308, CT2000 Conference (Como). \MR{1897816
  (2003f:18005)}

\bibitem[Fut04]{futia}
Carl~A. Futia, \emph{Weak omega categories {I}},
  \url{https://arxiv.org/abs/math/0404216} (2004).

\bibitem[Gur09]{gurski}
Nick Gurski, \emph{Nerves of bicategories as stratified simplicial sets}, J.
  Pure Appl. Algebra \textbf{213} (2009), no.~6, 927--946. \MR{2498786
  (2009m:18013)}

\bibitem[Hir03]{Hirschhorn}
Philip~S. Hirschhorn, \emph{Model categories and their localizations},
  Mathematical Surveys and Monographs, vol.~99, American Mathematical Society,
  Providence, RI, 2003. \MR{1944041}

\bibitem[HNP18]{HNP}
Yonatan {Harpaz}, Joost {Nuiten}, and Matan {Prasma}, \emph{{Quillen cohomology
  of $(\infty,2)$-categories}}, \url{https://arxiv.org/abs/1802.08046} (2018).

\bibitem[Joy97]{JoyalDisks}
Andr{\'e} Joyal, \emph{{Disks, Duality and $\Theta$-categories}}, available at
  \url{https://ncatlab.org/nlab/files/JoyalThetaCategories.pdf}, 1997.

\bibitem[Joy07]{Joyal2007}
\bysame, \emph{Quasi-categories vs simplicial categories}, preprint available
  at
  \url{http://www.math.uchicago.edu/~may/IMA/Incoming/Joyal/QvsDJan9(2007).pdf},
  2007.

\bibitem[Joy08]{JoyalVolumeII}
\bysame, \emph{The theory of quasi-categories and its applications}, preprint
  available at
  \url{http://mat.uab.cat/~kock/crm/hocat/advanced-course/Quadern45-2.pdf},
  2008.

\bibitem[Lac02]{lack1}
Stephen Lack, \emph{{A Quillen model structure for $2$-categories}}, K-theory
  \textbf{26} (2002), no.~2, 171--205.

\bibitem[Lac04]{lack2}
\bysame, \emph{{A Quillen model structure for bicategories}}, K-theory
  \textbf{33} (2004), no.~3, 185--197.

\bibitem[Lac10]{LackCompanion}
\bysame, \emph{A 2-categories companion}, Towards higher categories, Springer,
  2010, pp.~105--191.

\bibitem[LMW10]{LMW}
Yves Lafont, Fran{\c{c}}ois M{\'e}tayer, and Krzysztof Worytkiewicz, \emph{A
  folk model structure on omega-cat}, Advances in Mathematics \textbf{224}
  (2010), no.~3, 1183--1231.

\bibitem[Lur09a]{htt}
Jacob Lurie, \emph{Higher topos theory}, Annals of Mathematics Studies, vol.
  170, Princeton University Press, Princeton, NJ, 2009. \MR{2522659}

\bibitem[Lur09b]{lurieGoodwillie}
\bysame, \emph{\textnormal{$(\infty, 2)$-categories and the Goodwillie Calculus
  I}}, \url{https://arxiv.org/abs/0905.0462} (2009).

\bibitem[OR18]{or}
Viktoriya Ozornova and Martina Rovelli, \emph{Model structures for
  $(\infty,n)$-categories on (pre)stratified simplicial sets and prestratified
  simplicial spaces}, \url{https://arxiv.org/abs/1809.10621} (2018).

\bibitem[Rez96]{RezkCat}
Charles Rezk, \emph{A model category for categories}, available at
  \url{https://faculty.math.illinois.edu/~rezk/cat-ho.dvi}, 1996.

\bibitem[Rie18]{EmilyNotes}
Emily Riehl, \emph{Complicial sets, an overture}, 2016 {MATRIX} annals, MATRIX
  Book Ser., vol.~1, Springer, Cham, 2018, available at \url{
  https://arxiv.org/abs/1610.06801}, pp.~49--76. \MR{3792516}

\bibitem[RV18]{RiehlVerityBook}
Emily {Riehl} and Dominic Verity, \emph{Elements of $\infty$-category theory},
  \url{http://www.math.jhu.edu/~eriehl/elements.pdf}, retrieved in Jan 2019,
  2018.

\bibitem[Str87]{StreetOrientedSimplexes}
Ross Street, \emph{The algebra of oriented simplexes}, J. Pure Appl. Algebra
  \textbf{49} (1987), no.~3, 283--335. \MR{920944}

\bibitem[Ver08a]{VerityComplicialAMS}
Dominic Verity, \emph{Complicial sets characterising the simplicial nerves of
  strict {$\omega$}-categories}, Mem. Amer. Math. Soc. \textbf{193} (2008),
  no.~905, xvi+184. \MR{2399898}

\bibitem[Ver08b]{VerityComplicialI}
\bysame, \emph{Weak complicial sets. {I}. {B}asic homotopy theory}, Adv. Math.
  \textbf{219} (2008), no.~4, 1081--1149. \MR{2450607}

\end{thebibliography}
\end{document}